\documentclass[journal,twoside,web]{IEEEtran}

\usepackage{hyperref}%
\hypersetup{colorlinks=true, linkcolor=blue, breaklinks=true, urlcolor=blue, citecolor=blue}

\usepackage{amsmath}
\usepackage{amssymb}

\usepackage{amsthm}
\usepackage[svgnames]{xcolor}
\usepackage{amsfonts}
\usepackage{centernot}
\usepackage{dsfont}
\usepackage{enumerate}
\usepackage{graphicx}
\usepackage{doi,xspace}
\usepackage{stmaryrd}

\newcommand{\eps}{\epsilon}

\usepackage{version}

\includeversion{arxiv}\excludeversion{acc}\excludeversion{extra}

\newcommand*{\mydoi}[1]{\href{http://dx.doi.org/#1}{\includegraphics[width=.75em]{doi.png}}}

\newcommand{\suchthat}{\;\ifnum\currentgrouptype=16 \middle\fi|\;}
\newcommand{\until}[1]{\{1,\dots, #1\}}
\newcommand{\subscr}[2]{#1_{\textup{#2}}}

\newcommand{\setdef}[2]{\{#1 \; | \; #2\}}

\newcommand{\map}[3]{#1: #2 \rightarrow #3}

\newcommand{\real}{\mathbb{R}}
\newcommand{\realpositive}{\mathbb{R}_{>0}}
\newcommand{\realp}{\mathbb{R}_{>0}}
\newcommand{\realnonnegative}{\mathbb{R}_{\geq0}}

\newcommand{\scirc}{\raise1pt\hbox{$\,\scriptstyle\circ\,$}}

\newcommand\oprocendsymbol{\hbox{$\square$}}
\newcommand\oprocend{\relax\ifmmode\else\unskip\hfill\fi\oprocendsymbol}

\newcommand{\imag}{\mathrm{i}}

\definecolor{Gray}{gray}{0.9}

\usepackage[prependcaption,colorinlistoftodos]{todonotes}

\DeclareSymbolFont{bbold}{U}{bbold}{m}{n}
\DeclareSymbolFontAlphabet{\mathbbold}{bbold}
\newcommand{\vect}[1]{\mathbbold{#1}}
\newcommand{\vectorones}[1][]{\vect{1}_{#1}}
\newcommand{\vectorzeros}[1][]{\vect{0}_{#1}}

\newcommand{\ds}{\displaystyle}

\newcounter{saveenum}

\newtheorem{theorem}{Theorem}
\newtheorem{proposition}[theorem]{Proposition}
\newtheorem{lemma}[theorem]{Lemma}
\newtheorem{corollary}[theorem]{Corollary}
\newtheorem{definition}[theorem]{Definition}
\newtheorem{example}[theorem]{Example}
\newtheorem{assumption}{Assumption}
\newtheorem{remark}[theorem]{Remark}

\graphicspath{{./images/}{./fig}}

\newcommand{\jac}[1]{D\mkern-0.75mu{#1}}

\newcommand{\WeakP}[2]{\left\llbracket{#1}, {#2}\right\rrbracket}

\newcommand{\seminorm}[1]{{\left\vert\kern-0.25ex\left\vert\kern-0.25ex\left\vert #1
		\right\vert\kern-0.25ex\right\vert\kern-0.25ex\right\vert}}

\newcommand{\semimeasure}[1]{\mu_{\seminorm{\cdot}}\kern-0.5ex\left(#1\right)}

\newcommand{\osL}{\operatorname{osL}}

\newcommand{\conv}{\operatorname{conv}}

\newcommand{\spectrum}{\operatorname{spec}}
\DeclareMathOperator*{\esssup}{ess\,sup}

\newcommand{\norm}[2]{\|#1\|_{#2}}

\DeclareMathOperator*{\argmax}{arg\,max}

\newcommand{\fPer}{\subscr{f}{P}}
\newcommand{\slope}[2]{\mathrm{slope}[#1,#2]}

\DeclareSymbolFont{bbold}{U}{bbold}{m}{n}
\DeclareSymbolFontAlphabet{\mathbbold}{bbold}

\renewcommand{\theenumi}{(\roman{enumi})}

\usepackage{enumitem}

\newcommand{\mcP}{\ensuremath{\mathcal{P}}}

\newcommand{\mcH}{\ensuremath{\mathcal{H}}}
\newcommand{\mcMH}{\ensuremath{\mathcal{MH}}}
\newcommand{\mcTH}{\ensuremath{\mathcal{TH}}}
\newcommand{\mcLDS}{\ensuremath{\mathcal{LDS}}}

\title{Contractivity of Recurrent Neural Networks:\\ A Non-Euclidean Polytopic Perspective}
\title{Non-Euclidean Contractivity of Continuous-Time Neural Networks}
\title{Non-Euclidean Contraction Analysis of Continuous-Time Neural Networks}
\author{Alexander Davydov,~\IEEEmembership{Graduate~Student~Member,~IEEE}, Anton V. Proskurnikov,~\IEEEmembership{Senior~Member,~IEEE}, \\ and Francesco Bullo,~\IEEEmembership{Fellow,~IEEE}
\thanks{This material is based upon work supported by the National Science Foundation Graduate Research Fellowship under Grant No.~2139319 and AFOSR grant FA9550-22-1-0059. Anton Proskurnikov is supported by the project 2022K8EZBW “Higher-order interactions in social dynamics with application to monetary networks”, funded by European Union – Next Generation EU within the PRIN 2022 program (D.D. 104 - 02/02/2022 MUR).}
\thanks{Alexander Davydov and Francesco Bullo are with the Department
	of Mechanical Engineering and the Center for Control, Dynamical Systems,
	and Computation, University of California, Santa Barbara, 93106-5070, USA.
	{\tt\{davydov, bullo\}@ucsb.edu}. }
\thanks{Anton V. Proskurnikov is with the Department of Electronics and
	Telecommunications, Politecnico di Torino, Turin, Italy
	~\tt{anton.p.1982@ieee.org}.}
}

\newcommand{\metzler}[1]{\subscr{\lceil#1\rceil}{Mzr}}
\newcommand{\lognorm}[2]{\mu_{#2}(#1)}

\newcommand{\dmin}{\subscr{d}{1}}
\newcommand{\dmax}{\subscr{d}{2}}

\newcommand{\fHNN}{\subscr{f}{H}}
\newcommand{\fFR}{\subscr{f}{FR}}
\newcommand{\fMLure}{\subscr{f}{ML}}

\begin{document}
\maketitle



\begin{abstract}
	Critical questions in dynamical neuroscience and machine learning are
	related to the study of continuous-time neural networks and their stability,
	robustness, and computational efficiency. These properties can be
	simultaneously established via a contraction analysis.
	
	This paper develops a comprehensive non-Euclidean contraction theory for continuous-time
	neural networks.
	Specifically, we provide novel sufficient conditions for the contractivity of 
	general classes of continuous-time
	neural networks including Hopfield, firing rate, Persidskii, Lur'e, and other
	neural networks with respect to the non-Euclidean
	$\ell_1/\ell_\infty$ norms.
	These sufficient conditions are based upon linear programming or, in some
	special cases, establishing the Hurwitzness of a particular Metzler matrix.
	To prove these sufficient conditions, we develop novel results on
	non-Euclidean logarithmic norms and a novel necessary and sufficient
	condition for contractivity of systems with locally Lipschitz dynamics.
	For each model, we apply our theoretical results to compute the optimal contraction rate and corresponding
		weighted non-Euclidean norm with respect to which the neural network is contracting.
\end{abstract}



\thispagestyle{empty}
\pagestyle{empty}

\section{Introduction}

\textbf{Motivation from dynamical neuroscience and machine learning.}
Tremendous progress made in neuroscience research has produced new
understanding of biological neural processes. Similarly, machine learning
has become a key technology in modern society, with remarkable progress in
numerous computational tasks.  Much ongoing research focuses on artificial
learning systems inspired by neuroscience that (i) generalize better, (ii)
learn from fewer examples, and (iii) are increasingly energy-efficient. We
argue that further progress in these disciplines hinges upon modeling,
analysis, and computational challenges, some of which we highlight in what follows.

In \textbf{dynamical neuroscience}, several continuous-time neural network (NN)
models are widely studied, including membrane potential models such as the
Hopfield neural network~\cite{JJH:84} and firing-rate
models~\cite{KDM-FF:12}.  Clearly, such models are simplifications of
complex neural dynamics.
For example, if $f(x)$ is an NN model of a neural circuit, the true
dynamics may be better described by
\begin{equation}
  \dot{x}(t) = f(x(t)) + g(x(t), x(t-\tau(t))),
\end{equation}
where $g$ captures model uncertainty and time-delays.  In other words,
to account for uncertainty in the system, the nominal
dynamics $f(x)$ must exhibit robust stability with respect to unmodeled
dynamics and delays.
Additionally, central pattern generators (CPGs) are biological neural circuits that
generate periodic signals and are the source of rhythmic motor behaviors
such as walking and swimming.
To properly model CPGs in NNs, a computational neuroscientist would need
to ensure that, if an NN is interconnected with a CPG, then
all trajectories of the NN converge to a
unique stable limit cycle.

\textbf{Machine learning} scientists have widely adopted discrete-time NNs
for pattern recognition and analysis of sequential data and much recent
interest~\cite{SB-JZK-VK:19,AK-ZZ-VS:20,MR-RW-IRM:20,SJ-AD-AVP-FB:21f}
has focused on the closely-related class of implicit NNs. In
particular, training implicit networks corresponds to solving fixed-point
problems of the form
\begin{equation}\label{eq:fixedpoint}
  x = \Phi(Ax + Bu + b),
\end{equation}
where $x$ is the neural state variable, $\Phi$ is an activation function,
$A$ and $B$ are synaptic weights, $u$ is the input stimulus, and $b$ is a
bias term. Note that (i) the fixed point in equation~\eqref{eq:fixedpoint}
is the equilibrium point of the continuous-time NN $\dot{x} = -x + \Phi(Ax + Bu + b)$, (ii)
the training problem requires the efficient computation of gradients of a
given loss function with respect to model parameters; in turn, this
computation can be cast again as a fixed-point problem. In other words, in
the design of implicit NNs, it is essential to pick
model weights in such a way that fixed-point equations have
unique solutions for all possible inputs and activation functions, and
fixed-points and corresponding gradients can be computed
efficiently.

Finally, an additional challenge facing machine learning scientists is
robustness to adversarial perturbations. Indeed, it is
well-known~\cite{CZ-WZ-IS-JB-DE-IG-RF:13} that artificial deep NNs are 
sensitive to adversarial perturbations: small input changes
may lead to large output changes and loss in pattern recognition accuracy.
One proposed remedy is to characterize the Lipschitz constants of these
networks and use them as regularizers in the training process. This remedy
leads to certifiable robustness bounds with respect to adversarial
perturbations~\cite{MR-RW-IRM:21,MF-MM-GJP:20}. In short,
the input/output Lipschitz constants of NNs need to be tightly estimated,
e.g., in the context of the fixed-point equation~\eqref{eq:fixedpoint}.

\textbf{A contraction theory for neural networks.}  Motivated by the
challenges arising in neuroscience and machine learning, this paper aims to
perform a \emph{robust stability analysis} of continuous-time NNs and
develop \emph{optimization methods} for discrete-time NN models.
Serendipitously, both these objectives can be simultaneously achieved
through a contraction analysis for the NN dynamics.

For concreteness' sake, we briefly review how the aforementioned challenges
are addressed by a contraction analysis.  Infinitesimally
contracting dynamics enjoy highly ordered \emph{transient} and
\emph{asymptotic} behaviors:
(i) initial conditions are forgotten and a certain distance between
trajectories is monotonically vanishing~\cite{WL-JJES:98},
(ii) time-invariant systems admit a unique globally exponentially
stable equilibrium with two natural Lyapunov functions (distance from the
equilibrium and norm of the vector field)~\cite{WL-JJES:98},
(iii) periodic systems admit a unique globally exponentially stable
periodic solution or, for systems with periodic inputs, each solution
entrains to the periodic input~\cite{GR-MDB-EDS:10a},
(iv) contracting vector fields enjoy highly
robust behavior, e.g., see \cite{HT-SJC-JJES:21,AD-SJ-FB:20o}, including
(a) input-to-state stability,
(b) finite input-state gain,
(c) contraction margin with respect to unmodeled dynamics, and
(d) input-to-state stability under delayed dynamics.
Hence, the contraction rate is a natural measure/indicator of robust
stability.


Regarding computational efficiency, our recent
work~\cite{FB-PCV-AD-SJ:21e,SJ-AD-AVP-FB:21f} shows how to design efficient
fixed-point computation schemes for contracting systems (with respect to
arbitrary and non-Euclidean $\ell_1/\ell_\infty$ norms) in the style of
monotone operator theory~\cite{EKR-WY:21}.
Specifically, for contracting dynamics with respect to a
diagonally-weighted $\ell_1/\ell_\infty$ norm, optimal step-sizes and
convergence factors are given in~\cite[Theorem~2]{SJ-AD-AVP-FB:21f}. These
results are directly applicable to the computation of fixed-points in
implicit neural networks, as in equation~\eqref{eq:fixedpoint}.  These
step-sizes, however, depend on the contraction rate. Therefore, optimizing
the contraction rate of the dynamics directly improves the convergence
factor of the corresponding discrete algorithm.

\textbf{Literature review.}
The dynamical properties of continuous-time NN models have
been studied for several decades. Shortly after Hopfield's original
work~\cite{JJH:84}, control-theoretic ideas were proposed
in~\cite{ANM-JAF-WP:89}. Later, \cite{EK-AB:94,MF-SM-MM:94,MF-AT:95}
obtained various version of the following result: Lyapunov diagonal
stability of the synaptic matrix is sufficient, and in some cases
necessary, for the existence, uniqueness, and global asymptotic stability
of the equilibrium. 
More recently, \cite{EN-JC:21a} studies linear-threshold rate neural
dynamics, where activation functions are piecewise-affine; it is shown that
the dynamics have a unique equilibrium if and only if the synaptic matrix
is a $\mcP$-matrix, a weaker condition than Lyapunov diagonal stability. Since checking this condition is NP-hard, more
conservative conditions are provided as well.
Beyond Lyapunov diagonal stability and $\mcP$-matrices, \cite{YF-TGK:96} is the earliest reference on
the application of logarithmic norms and contraction-theoretic principles to Hopfield neural networks and
provides results on $\ell_p$ logarithmic norms of the Jacobian for networks
with smooth activation functions. Alternatively, \cite{SA:02} proposes a quasi-dominance
condition on the synaptic matrix (in lieu of Lyapunov diagonal stability).
Finally, similar to non-Euclidean contraction, \cite{HQ-JP-ZBX:01} proposes the notion of the nonlinear measure of a
map to study global asymptotic stability; this notion is closely related to
the $\ell_1$ one-sided Lipschitz constant of the Hopfield neural network vector field. A comprehensive
survey on stability criteria for continuous-time NNs is available in~\cite{HZ-ZW-DL:14}.

The importance of non-Euclidean log norms in contraction theory is
highlighted, for example, in~\cite{GR-MDB-EDS:10a,ZA-EDS:14}.
In the spirit of these works, the non-Euclidean contractivity of monotone Hopfield neural
networks is studied in~\cite{SJ-AD-FB:20r}; see also~\cite{VC-FB-GR:22g} for the non-Euclidean contractivity of Hopfield
neural networks undergoing Hebbian learning. 

Finally, Euclidean contractivity of continuous-time NNs has been
studied, e.g., see the early reference~\cite{YF-TGK:96}, the related
discussion in~\cite{MR-RW-IRM:20}, and the recent
work~\cite{LK-ME-JJES:22}.



\textbf{Contributions.}  This paper contributes fundamental
control-theoretic understanding to the study of artificial neural networks
in machine learning and neuronal circuits in neuroscience, thereby building
a hopefully useful bridge among these three disciplines.

Specifically, the paper develops a comprehensive contraction theory for classes of continuous-time NN
models. In order to develop this theory, we make several technical contributions on non-Euclidean logarithmic norms and nonsmooth contraction theory. To be specific, first, we obtain novel
logarithmic norm results including (i) the quasiconvexity of the $\ell_1$
and $\ell_\infty$ logarithmic norms with respect to diagonal weights and
provide novel optimization techniques to compute optimal weights which
yield larger contraction rates, (ii) logarithmic norm properties of
principal submatrices of a matrix with respect to monotonic norms, and
(iii) explicit formulas for the $\ell_1$ and $\ell_\infty$ logarithmic
norms under multiplicatively-weighted uncertainty, resulting in a
maximization of the logarithmic norm over a matrix polytope. The matrix
polytopes described in (iii) are of special interest since the Jacobian matrix of
the Hopfield or firing-rate neural network vector field always lies inside this polytope. The formulas
in~(iii) generalize previous results \cite[Theorem~3.8]{YF-TGK:96},
\cite[Lemma~3]{WH-JC:09} and \cite[Lemma~8]{SJ-AD-AVP-FB:21f}. 

Motivated by our non-Euclidean logarithmic norm results, we define
$M$-Hurwitz matrices, i.e., matrices whose Metzler majorant is Hurwitz. We
compare $M$-Hurwitz matrices with other
classes of
matrices including quasidominant, totally Hurwitz, and
Lyapunov diagonally stable matrices.

Second, we provide a nonsmooth extension to contraction theory. We show
that, for locally Lipschitz vector fields, the one-sided Lipschitz constant
is equal to the essential supremum of the logarithmic norm of the Jacobian.
This equality allows us to use our novel logarithmic norm results and apply
them to NNs that have nonsmooth activation functions.

Finally, we apply our theoretical developments as we establish conditions for the non-Euclidean contractivity of multiple classes of recurrent neural circuits
and nonlinear dynamical models, including Hopfield, firing rate,
Persidskii, Lur'e, and others. We consider locally Lipschitz activation functions
that satisfy an inequality of the form
$\ds\dmin \leq \frac{\phi(x) - \phi(y)}{x - y} \leq \dmax,$ for all $x \neq y \in \real,$
where $\dmin$ may be negative and $\dmax$ may be infinite. 
Indeed, the 
importance of nonmonotonic activation functions is discussed in~\cite{MM:93}.
This class of activation functions is more general than all of the continuous activation functions
mentioned in~\cite[Section~II.B]{HZ-ZW-DL:14}. Thus, our non-Euclidean
contraction framework allows for a more systematic framework for the analysis
of these classes of NNs with fewer restrictions on the activation functions.
For each model, we propose a linear program to characterize the optimal
contraction rate and corresponding weighted non-Euclidean $\ell_1$ or
$\ell_\infty$ norm. In some special cases, we show that the linear program
reduces to checking an $M$-Hurwitz condition.
Our results simplify the computation of a common Lyapunov function over a
polytope with $2^n$ vertices to a simple condition involving just $2$ of its
vertices or, in some cases, all the way to a closed form expression.

For each model, 
we demonstrate that the dynamics enjoy strong, absolute and
total contractivity properties. In the spirit of absolute and connective
stability, absolute contractivity means that the dynamics are contracting
independently of the choice of activation function and connective stability
means that the dynamics remain contracting whenever edges between neurons
are removed. Total contractivity means that if the synaptic matrix is
replaced by any principal submatrix, the
resulting dynamics remain contracting. The process of replacing the nominal NN
with a subsystem NN is referred to as ``pruning'' both in neuroscience and
in machine learning.

A preliminary version of this work appeared in~\cite{AD-AVP-FB:21k}. Compared to~\cite{AD-AVP-FB:21k}, this version (i) includes proofs of all technical results, (ii) provides closed-form worst-case log norms over a larger class of matrix polytopes in Lemma~\ref{lemma:affine-scaling}, (iii) studies a more general class of locally Lipschitz activation functions in Section~\ref{sec:neuraldynamics}, allowing for both nonmonotonic activation functions as well as activations that have unbounded derivative, (iv) has a complete characterization of contractivity of Hopfield and firing-rate neural networks with respect to both $\ell_1$ and $\ell_\infty$ norms, (v) provides a novel sufficient (and nearly necessary) condition for the non-Euclidean contractivity of a Lur'e model with multiple nonlinearities in Theorem~\ref{theorem:MLure-MILP}, and (vi) includes additional comparisons to Euclidean contractivity conditions in Remark~\ref{rmk:SlotineMetzlerPlus} and to Lyapunov diagonal stability in Section~\ref{sec:preview}.


\textbf{Notation.} For a set $S$, we let $S^n$ be the Cartesian product of $n$ copies of $S$, $|S|$ be its cardinality, and, if $S \subseteq \real^n$, $\conv S$ be the convex hull of $S$. For two matrices $A,B$, we let $A_{ij}$ be the entry in the $i$-th row and $j$-th column of $A$, $A \circ B$ be entrywise
multiplication and $|A|$ be the entrywise absolute value. For $p \in [1,\infty]$, we let $\|\cdot\|_p$ denote the $\ell_p$ norm, i.e., for a vector $x \in \real^n$, $\|x\|_p = (\sum_{i=1}^n |x_i|^p)^{1/p}$ if $p \in {[1,\infty[}$ and $\|x\|_\infty = \max_{i \in \until{n}} |x_i|$. For an invertible matrix $R \in \real^{n \times n}$, we define the $R$-weighted $\ell_p$ norm by $\|x\|_{p,R} = \|Rx\|_p$. Vector inequalities of the form $x \leq y$ are
entrywise. For a vector $\eta \in \real^n$, we define $[\eta] \in \real^{n
	\times n}$ to be the diagonal matrix with diagonal entries equal to
$\eta$.  We let $\vectorones[n],\vectorzeros[n] \in \real^n$ be the
all-ones and all-zeros vectors, respectively. We say a norm $\|\cdot\|$ on
$\real^n$ is \emph{monotonic} if for all $x,y \in \real^n$, $|x| \leq |y|
\implies \|x\| \leq \|y\|$. A matrix $M \in \real^{n \times n}$ is \emph{Metzler} if $M_{ij}
\geq 0$ for all $i \neq j$.  For a matrix $A \in \real^{n\times n}$, its
\emph{spectral abscissa} is $\alpha(A)=\max\setdef{\mathrm{Re}(\lambda)}
{\lambda\in\spectrum(A)},$ where $\mathrm{Re}(\lambda)$ denotes the real part of $\lambda$, and its \emph{Metzler majorant}
$\metzler{A}\in\real^{n\times{n}}$ is defined by $(\metzler{A})_{ij} =
\begin{cases} A_{ii},\quad & \text{if }i=j \\ |A_{ij}|, \quad &\text{if } i\neq j
\end{cases}$.

\section{Preview of main contractivity results and advantages of a non-Euclidean analysis}\label{sec:preview}
To motivate the mathematical tools and analysis in Sections~\ref{sec:novel-mm}-\ref{sec:neuraldynamics}, we will showcase the main contractivity results for Hopfield and firing-rate neural networks under simplifying assumptions to provide a baseline for comparison to other standard stability conditions for these classes of neural networks. 

The continuous-time Hopfield and firing-rate neural networks are the following two dynamical systems:
\begin{align}
	\dot{x} &= -Cx + A\Phi(x) + u =: \fHNN(x), \label{eq:Hopfield-simplified}\\
	\dot{x} &= -Cx + \Phi(Ax + u) =: \fFR(x), \label{eq:FR-simplified}
\end{align}
where $x \in \real^n$ is the state of the neural network (either a vector of membrane potentials or firing rates), $C \in \real^{n \times n}$ is a positive semidefinite diagonal matrix of dissipation rates, $A \in \real^{n \times n}$ is the synaptic matrix , $u \in \real^n$ is a constant external stimulus, and $\map{\Phi}{\real^n}{\real^n}$ is an activation function which satisfies $\Phi(x) = (\phi_1(x_1),\dots,\phi_n(x_n))$. In the machine learning literature, such continuous-time NNs have been given the name neural ODEs~\cite{RTQC-YR-JB-DD:18}.

For exposition's sake, we make the following standing assumptions throughout the remainder of this section: 
\begin{assumption}\label{assm-simplified}
	\begin{enumerate}
		\item\label{assm:simpl-1} $C = I_n$,
		\item\label{assm:simpl-2} the matrix $\metzler{A}$ is irreducible, and
		\item\label{assm:simpl-3} each $\phi_i$ is continuously differentiable and satisfies $\\ 0 \leq \phi_i'(x) \leq 1$ for all $x \in \real$.
	\end{enumerate}
\end{assumption}
Under these assumptions, we can state our main results compactly:
\begin{proposition}\label{prop:preview}
	Consider the Hopfield and firing-rate neural networks~\eqref{eq:Hopfield-simplified} and~\eqref{eq:FR-simplified} satisfying Assumption~\ref{assm-simplified}, suppose $\alpha(\metzler{A}) < 1$, and define $c = 1 - \max\{\alpha(\metzler{A},0)\}$. Then
	\begin{enumerate}
		\item\label{item:prop11} the Hopfield neural network is contracting with rate $c > 0$, i.e., any two trajectories $x_1(\cdot), x_2(\cdot)$ of~\eqref{eq:Hopfield-simplified} satisfy
		\begin{equation*}
			\|x_1(t) - x_2(t)\|_{1,[\eta]} \leq e^{-ct}\|x_1(0) - x_2(0)\|_{1,[\eta]},
		\end{equation*}
		for all $t \geq 0$, where $\eta \in \realpositive^n$ is the dominant left eigenvector of the Metzler matrix $\metzler{A}$.		
		\item\label{item:prop12} the firing-rate neural network is contracting with rate $c > 0$, i.e., any two trajectories $x_1(\cdot), x_2(\cdot)$ of~\eqref{eq:FR-simplified} satisfy
		\begin{equation*}
			\|x_1(t) - x_2(t)\|_{\infty,[\xi]^{-1}} \leq e^{-ct}\|x_1(0) - x_2(0)\|_{\infty,[\xi]^{-1}},
		\end{equation*}
		for all $t \geq 0$ where $\xi \in \realpositive^n$ is the dominant right eigenvector of the Metzler matrix $\metzler{A}$.
	\end{enumerate}
\end{proposition}

In particular, under Assumption~\ref{assm-simplified} and $\alpha(\metzler{A}) < 1$ (or equivalently $\alpha(\metzler{-I_n + A}) < 0$), for each $u \in \real^n$, both the Hopfield and firing-rate neural networks have unique globally exponentially stable equilibria and thus the condition $\alpha(\metzler{-I_n + A}) < 0$ provides a  novel sufficient condition for the existence of a unique globally exponential stable equilibrium along with many additional robustness properties offered by contracting systems such as robustness to uncertainties and entrainment to periodic inputs. 

Although in this paper we primarily study the continuous-time NNs~\eqref{eq:Hopfield-simplified} and~\eqref{eq:FR-simplified}, we remark that many results apply to classes of discrete-time NNs as well. Specifically, given a continuous-time NN, $\dot{x} = \subscr{f}{NN}(x)$, which is contracting, the forward Euler discretization of the continuous-time NN with stepsize $h > 0$ yields a residual neural network
\begin{equation}
	x_{k+1} = x_k + h \subscr{f}{NN}(x_k),
\end{equation}
which is contracting in the sense of the Banach fixed point theorem for sufficiently small $h$ (see, e.g.,~\cite[Theorem~8]{FB-PCV-AD-SJ:21e}). For recent results on contraction for a different class of discrete-time NNs, we refer to~\cite{MR-IM:20}.

The condition $\alpha(\metzler{-I_n + A}) < 0$ is different from the well-known result that Lyapunov diagonal stability (LDS) of $-I_n+A$, i.e., existence of a vector $\eta \in \realpositive^n$ satisfying 
\begin{equation}\label{eq:LDS-Hopfield}
	[\eta](-I_n + A) + (-I_n +A)^\top [\eta] \prec 0,
\end{equation} implies the existence of a unique globally asymptotically stable equilibrium point for the Hopfield neural network~\cite{MF-AT:95}. Moreover, the condition $\alpha(\metzler{-I_n + A}) < 0$ is stronger than LDS of $-I_n + A$, which we prove in Lemma~\ref{lemma:matrixclasses}, yet it implies the stronger property of contractivity.

Beyond LDS, an alternative way to establish the stability of the neural networks~\eqref{eq:Hopfield-simplified} and~\eqref{eq:FR-simplified} is via absolute stability analysis of Lur'e systems and methods via quadratic Lyapunov functions. These methods are typically based upon linear matrix inequalities (LMIs), see, e.g.,~\cite{LDA-MC:13,MG-VA-ST-DA:22} and the discussion in~\cite[Section~I.V.]{HZ-ZW-DL:14}. Compared to these classical approaches, establishing contractivity with respect to diagonally-weighted $\ell_1$ or $\ell_\infty$ norms provides both computational and practical advantages, which we highlight in the following paragraphs.

\textbf{Computational benefits.} In the non-Euclidean contraction analysis of many classes of neural networks, contractivity is checked either via linear programming or, in some simpler instances, the stability of appropriate Metzler matrices. As argued in~\cite{AR:15}, from a computational point of view, both of these tests are more scalable than LMIs are. Indeed, there exist efficient algorithms for computing Perron eigenvalues and eigenvectors for irreducible Metzler matrices~\cite{PVA:91}.

\textbf{Practical benefits.} Compared to stability with respect to a quadratic Lyapunov function, there are also practical advantages to establishing contractivity with respect to diagonally-weighted $\ell_1$ and $\ell_\infty$ norms. These benefits include (i) the $\ell_1$ norm (respectively, the $\ell_\infty$ norm) is well suited for systems with conserved quantities (respectively, systems with translation invariance), e.g., see the theory of weakly contracting and monotone systems in~\cite[Chapter~4]{FB:23-CTDS}, (ii) in machine learning, analysis of the adversarial robustness of a NN often needs to be performed in a non-Euclidean norm, because NNs are known to be vulnerable to small disturbances as measured in the $\ell_\infty$ norm~\cite{IJG-JS-CZ:15}, and (iii) contractivity with respect to non-Euclidean norms ensures robustness with respect to edge removals and structural perturbations, e.g., see the notion of connective stability in~\cite{DDS:78}.

To elaborate on point (iii) in the previous paragraph, in continuous-time NNs such as the Hopfield and firing rate neural networks~\eqref{eq:Hopfield-simplified}-\eqref{eq:FR-simplified}, the synaptic matrix $A$ defines a graph structure whereby there is an outgoing synapse from neuron $j$ to neuron $i$ provided that $A_{ij} \neq 0$. As we will show in Corollary~\ref{cor:pruning}, if the neural network is contracting with respect to a diagonally-weighted $\ell_1$ or $\ell_\infty$ norm, it is \emph{connectively contracting}. Specifically, the removal of any edge\footnote{Removing an edge corresponds to zeroing a non-diagonal entry of $A$.} or neuron from the graph\footnote{Removing the $i$-th neuron corresponds to removing the $i$-th row and column of $A$.} yields a new neural network that remains contracting with a rate greater than or equal to the rate of contraction of the original neural network. Note that this property is not enjoyed by stability conditions requiring a matrix to be LDS. Indeed, for the Hopfield neural network~\eqref{eq:Hopfield-simplified}, consider
$$A = \begin{bmatrix}
	0 & -1 & 1 \\ 1 & 0 & 15 \\ -1 & -15 & 0
\end{bmatrix}, \quad \tilde{A} = \begin{bmatrix}
	0 & -1 & 1 \\ 1 & 0 & 15 \\ -1 & 0 & 0
\end{bmatrix}.$$ Note that $A$ satisfies~\eqref{eq:LDS-Hopfield} with $\eta = \vectorones[n]$ so $-I_n + A$ is LDS and thus the Hopfield neural network is stable. However, zeroing out $A_{32}$ yields $\tilde{A}$ which verifies $\alpha(-I_n + \tilde{A}) > 0$, so the resulting NN is not absolutely stable.

In the following sections, we introduce additional mathematical tools to prove Proposition~\ref{prop:preview} under assumptions weaker than those listed in Assumption~\ref{assm-simplified}. Specifically, we (i) relax item~\ref{assm:simpl-1} to $C$ which is diagonal and positive semidefinite, (ii) relax item~\ref{assm:simpl-2} to also study $\metzler{A}$ which may be reducible, and (iii) relax item~\ref{assm:simpl-3} to study nonsmooth activation functions which may be nonmonotonic and may have unbounded slope. See Theorems~\ref{thm:osL-neural},~\ref{thm:osL-firingrate}, and~\ref{theorem:unbounded-slope} for these results. Beyond the proof of a more general version of Proposition~\ref{prop:preview}, we also establish $\ell_\infty$ contractivity of the Hopfield neural network in Theorem~\ref{thm:osL-neural-inf}, the $\ell_1$ contractivity of the firing-rate neural network in Theorem~\ref{thm:osL-fr-1}, and study the contractivity of other classes of neural networks in Section~\ref{subsec:other-models}. In the interest of readability, we postpone proofs of most technical results to Appendix~\ref{app:proofs} and include proofs regarding contractivity of classes of NNs in the main body of the text.

\section{Review of relevant matrix analysis}\label{sec:lognorms}

\subsection{Log norms}

Let $\|\cdot\|$ be a norm on $\real^n$ and its corresponding induced norm
on $\real^{n \times n}$. The \emph{logarithmic norm} (also called log norm
or matrix measure) of a matrix $A \in \real^{n \times n}$ is
\begin{equation}\label{eq:lognorm}
\lognorm{A}{} := \lim_{h\to0^+} \frac{\|I_n + hA\| - 1}{h}.
\end{equation}
We refer to~\cite{CAD-HH:72} for a list of properties of log norms, which
include subadditivity, convexity, and $\alpha(A) \leq
\mu(A)$.  
It is known that the log norm corresponding to an $R$-weighted $\ell_p$ norm is
$\lognorm{A}{p,R} = \lognorm{RAR^{-1}}{p}$. For diagonally weighted
$\ell_1, \ell_\infty$, and $\ell_2$ norms with $\eta \in \realpositive^n$,
\begin{align*}
  \lognorm{A}{1,[\eta]} &= \max_{i\in\until{n}} A_{ii} + \sum\nolimits_{j=1,j\not=i}^n
  \frac{\eta_j}{\eta_i}|A_{ji}| \\
  &= \min\setdef{b \in \real}{\metzler{A}^\top \eta \leq b \eta}, \\
  \lognorm{A}{\infty,[\eta]^{-1}} &= \max_{i\in\until{n}} A_{ii} + \sum\nolimits_{j=1,j\not=i}^n
  \frac{\eta_j}{\eta_i}|A_{ij}| \\
  &= \min\setdef{b \in \real}{\metzler{A} \eta \leq b \eta}, \\
  \lognorm{A}{2,[\eta]^{1/2}}
  &=\min\setdef{b \in \real}{[\eta]A + A^\top [\eta] \preceq 2b[\eta]}.
\end{align*}

The following result is due to~\cite{JS-CW:62} and \cite[Lemma~3]{OP-MV:06}.

\begin{lemma}[Optimal diagonally-weighted log norms for Metzler matrices]
	\label{lemma:efficientlognorm-metzler}
	Given a Metzler matrix $M\in\real^{n\times{n}}$,
	$p\in[1,\infty]$, and $\delta>0$, define
	$\eta_{M,p,\delta}\in\realpositive^n$ by
	\begin{equation}
	\eta_{M,p,\delta} = \Bigg( \frac{w_1^{1/p}}{v_1^{1/q}} ,\dots ,
	\frac{w_n^{1/p}}{v_n^{1/q}} \Bigg),
	\end{equation}
	where $q\in[1,\infty]$ is defined by $1/p+1/q=1$ (with the convention
	$1/\infty=0$) and where $v$ and $w\in\realpositive^n$ are the right and
	left dominant eigenvectors of the irreducible Metzler matrix
	$M+\delta\vectorones[n]\vectorones[n]^\top$ (whose existence is guaranteed by the
	Perron-Frobenius Theorem).
	Then for each $\eps>0$ there exists $\delta > 0$ such that
	\begin{enumerate}
		\item $\alpha(M)\leq \lognorm{M}{p,[\eta_{M,p,\delta}]} \leq
		\alpha(M)+\eps$,
		\item if $M$ is irreducible, then $\alpha(M)=
		\lognorm{M}{p,[\eta_{M,p,0}]}$.
	\end{enumerate}
\end{lemma}

Lemma~\ref{lemma:efficientlognorm-metzler} also ensures that for Metzler
matrices $M \in \real^{n \times n}$, $\inf_{\eta \in \realpositive^n} \mu_{p,[\eta]}(M)
= \alpha(M)$ for every $p \in [1,\infty]$.


\subsection{Classes of matrices}

We say a matrix $A \in \real^{n \times n}$ is
\begin{enumerate}
	\item \emph{Hurwitz stable}, denoted by $A \in \mcH$, if $\alpha(A)<0$,
	\item \emph{totally Hurwitz}, denoted by $A \in \mcTH$, if all principal submatrices
	of $A$ are Hurwitz stable,
	\item \emph{Lyapunov diagonally stable (LDS)}, denoted by $A \in
          \mcLDS$, if there exists a $\eta\in\realp^n$ such that
          $\lognorm{A}{2,[\eta]^{1/2}} < 0$, and
	\item \emph{M-Hurwitz stable}, denoted by $A \in \mcMH$, if $\alpha(\metzler{A})<0$.
\end{enumerate}


A matrix $A \in \real^{n \times n}$ is \emph{quasidominant}~\cite{PJM:77}
if there exists a vector $\eta \in \realpositive^n$ such that $$\eta_i
A_{ii} > \sum\nolimits_{j=1,j\not=i}^n \eta_j|A_{ij}|, \quad \text{for all
} i \in \until{n}.$$ This is equivalent to $\metzler{-A}\eta <
\vectorzeros[n]$, which, in turn, is equivalent (see, for example,
\cite[Theorem~15.17]{FB:22}) to the inequality $\alpha(\metzler{-A}) < 0$,
i.e., $-A \in \mcMH$.



The following results are essentially known in the literature, but not
collected in a unified manner.

\begin{lemma}[Inclusions for classes of matrices]\label{lemma:matrixclasses}
	($A \in \mcMH$) implies ($A \in \mcLDS$), $(A \in \mcLDS)$ implies $(A \in \mcTH)$, and $(A \in \mcTH)$ implies $(A \in \mcH)$.
\begin{center}
	\includegraphics[width=0.99\columnwidth]{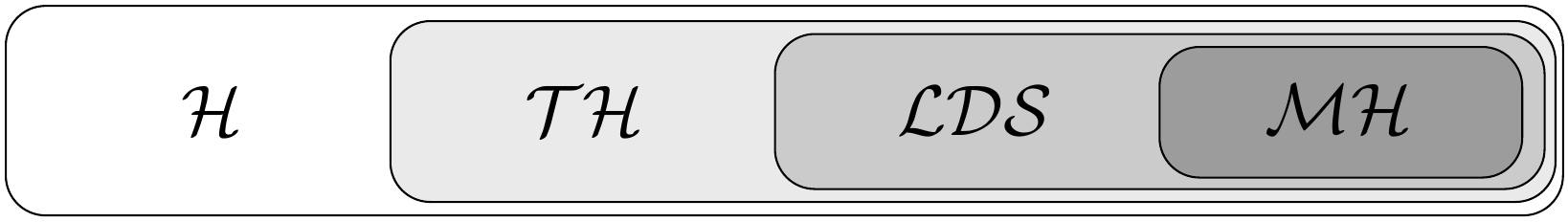}
\end{center}
\end{lemma}

We show that the counter-implications in Lemma~\ref{lemma:matrixclasses} do
not hold.

\begin{example}
	\begin{enumerate}
		\item ($A \in \mcLDS \centernot\implies A \in \mcMH$) The matrix $A = \left[\begin{smallmatrix} -1 & -1 \\ 2 & -1 \end{smallmatrix}\right]$
		satisfies $\mu_{2}(A) = -0.5,$ so $A \in \mcLDS$. However,
		$\alpha(\metzler{A}) = \sqrt{2} - 1 > 0$, so $A \notin \mcMH$.
		\item ($A \in \mcTH \centernot\implies A \in \mcLDS$) is proved in~\cite[Remark~4]{GPB-AB-RJP:78}.
		\item ($A \in \mcH \centernot\implies A \in \mcTH$) The matrix $A = \left[\begin{smallmatrix} 1 & 1 \\ -4 & -3 \end{smallmatrix}\right]$
		satisfies $\alpha(A) = -1$, so $A \in \mcH$. However, $A \notin \mcTH$ since it has a positive diagonal entry.
	\end{enumerate}
\end{example}

In the context of Proposition~\ref{prop:preview}, the condition $\alpha(\metzler{-I_n + A}) < 0$, is equivalent to asking $-I_n + A \in \mcMH$, which, in light of Lemma~\ref{lemma:matrixclasses}, implies $-I_n + A \in \mcLDS$, which was the previously known sufficient condition for asymptotic stability of a unique fixed point of the Hopfield NN.

\section{Novel log norm results}
\label{sec:novel-mm}

\subsection{Optimizing non-Euclidean log norms}
First, we provide novel results on optimizing diagonal weights for $\ell_1$
and $\ell_\infty$ log norms and provide computational methods to compute
these weights.

\begin{theorem}[Quasiconvexity of $\mu$ with respect to diagonal weights]\label{thm:quasiconvexity}
	For fixed $A \in \real^{n \times n}$,
	consider the maps from $\realpositive^n$ to $\real$ defined by
	\begin{equation}\label{eq:lognormmaps}
	\begin{aligned}
	&\eta \mapsto \mu_{1,[\eta]}(A), \qquad \eta \mapsto \mu_{\infty,[\eta]^{-1}}(A).
	\end{aligned}
	\end{equation}
	Then
	\begin{enumerate}
		\item\label{item:quasiconvexity} The maps in~\eqref{eq:lognormmaps} are continuous, quasiconvex, and their sublevel sets are polytopes.
		\item\label{item:bisection} Minimizing the maps in~\eqref{eq:lognormmaps} may be executed via the minimization problems
		\begin{equation}\label{eq:mu1optimization}
		\begin{aligned}
		\hspace{-\leftmargin}
		\inf_{b \in \real, \eta \in \realpositive^n} & \qquad b \\
		\text{s.t.}\qquad  & \metzler{A}^\top \eta \leq b\eta,
		\end{aligned}
		\end{equation}
		for $\mu_{1,[\eta]}(A)$ and
		\begin{equation}\label{eq:muinfoptimization}
		\begin{aligned}
		\hspace{-\leftmargin}
		\inf_{b \in \real, \eta \in \realpositive^n} & \qquad b \\
		\text{s.t.}\qquad  & \metzler{A} \eta \leq b\eta,
		\end{aligned}
		\end{equation}
		for $\mu_{\infty,[\eta]^{-1}}(A)$.
	\end{enumerate}
\end{theorem}

\begin{remark}\label{rem:minima}
  If $\metzler{A}$ is irreducible, by Lemma~\ref{lemma:efficientlognorm-metzler}
  the optimization problems in~\eqref{eq:mu1optimization} and~\eqref{eq:muinfoptimization} attain their minima so
  that the $\inf$ may be replaced by $\min$. Then the problems may be solved by
  a bisection on $b \in [-\|A\|, \|A\|]$, where each step of the
  algorithm is a linear program (LP) in $\eta$.
  
  Moreover, the minima in~\eqref{eq:mu1optimization} and~\eqref{eq:muinfoptimization} exist for many types of reducible matrices, e.g. when $\metzler{A}$ is a block-diagonal matrix whose diagonal blocks are irreducible.
  
  In the event that the minimum does not exist, let $b^\star$ be the infimum value of either \eqref{eq:mu1optimization} or~\eqref{eq:muinfoptimization}. Then for any $\epsilon > 0$, one can still apply the bisection algorithm to find a choice of $\eta$ such that $\mu_{[\eta]}(A) \leq b^\star + \epsilon$, where $\mu_{[\eta]}(\cdot)$ denotes either $\mu_{1,[\eta]}(\cdot)$ or $\mu_{\infty,[\eta]^{-1}}(\cdot)$.
\end{remark}

\begin{remark}
Notice that the sets of feasible vectors $\eta$ in~\eqref{eq:mu1optimization} and~\eqref{eq:muinfoptimization} are polyhedral cones, that is, if $\eta$ is feasible, then $\theta\eta$ is also feasible for all $\theta>0$. 
Hence, the constraint $\eta \in \realpositive^n$ can be replaced by an equivalent constraint $\eta\in {[\varepsilon,\infty[}^n$, 
where $\varepsilon>0$ is an arbitrary constant. This can be useful, because LP solvers usually handle problems with non-strict
inequalities.
\end{remark}

Next, we provide closed-form expressions for $\ell_1$ and $\ell_\infty$ log
norms over a certain polytopes of matrices. Polytopes of interest are
defined by a nominal matrix multiplied by a diagonally-weighted uncertainty
and shifted by an additive diagonal matrix. Such matrix polytopes arise in
tests verifying the contractivity of Hopfield and firing-rate NNs and will play a critical role in our analysis.
\newcommand{\od}{\overline{d}}

\twocolumn[ \hrulefill{\newline This insert corresponds to
	Lemma~\ref{lemma:affine-scaling}.  For $A\in\real^{n\times{n}}$,
	$c\in\real^{n}$, $\dmin\leq\dmax\in\real$, $\od = \max\{|\dmin|,|\dmax|\}$, and
	$\eta\in\realpositive^{n}$,
	\begin{align}
		\label{fact:ms:1} \max_{d\in[\dmin,\dmax]^n}\! \mu_{\infty,[\eta]} ([c]+[d]A)
		&=
		\max\big\{ \mu_{\infty,[\eta]} ([c]+\dmin A), \mu_{\infty,[\eta]} ([c]+\dmax A)\big\}, \\
		\label{fact:ms:2}  \max_{d\in[\dmin,\dmax]^n}\! \mu_{1,[\eta]} ([c]+A[d])
		&=
		\max\big\{ \mu_{1,[\eta]} ([c]+\dmin A), \mu_{1,[\eta]} ([c]+\dmax A)\big\},
		\\
		\label{fact:ms:3}
		\max_{d \in [\dmin,\dmax]^n} \!\!\mu_{\infty,[\eta]}([c] + A[d])
		&= \max\{\mu_{\infty,[\eta]}([c] + \od A - (\od - \dmin)(I_n \circ A)), \mu_{\infty,[\eta]}([c] + \od A - (\od - \dmax)(I_n \circ A))\}, \\
		\label{fact:ms:4}  \max_{d \in [\dmin,\dmax]^n}\! \mu_{1,[\eta]}([c] + [d]A)
		&=
		\max\{\mu_{1,[\eta]}([c] + \od A - (\od - \dmin)(I_n \circ A)), \mu_{1,[\eta]}([c] + \od A - (\od - \dmax)(I_n \circ A))\}.
	\end{align}
}
\hrulefill \\
]
\begin{lemma}[Max value of $\ell_1/\ell_\infty$ log norms under multiplicative scalings]
  \label{lemma:affine-scaling}
  Any $A\in\real^{n\times{n}}$, $c\in\real^{n}$,
  $\dmin\leq\dmax\in\real$, and $\eta\in\realpositive^{n}$ satisfy
  formulas~\eqref{fact:ms:1}-\eqref{fact:ms:4} where $\od = \max\{|\dmin|,|\dmax|\}$.
\end{lemma}

Recall that the log norm is a convex function and that the maximum value of
a convex function over a polytope is achieved at one of the vertices of the
polytope. In the special case in Lemma~\ref{lemma:affine-scaling},
formulas~\eqref{fact:ms:1}-\eqref{fact:ms:4} ensure that one needs to check
only $2$ vertices of the polytope, rather than $2^n$.

Finally, we show how the optimal diagonal weights that minimize the
maximum value of the log norm of a matrix polytope as in
Lemma~\ref{lemma:affine-scaling} can be easily computed.

\begin{corollary}\label{cor:minimax}
	Let $A, c, \dmin$, and $\dmax$ be as in Lemma~\ref{lemma:affine-scaling}. Then for $\mu_{[\eta]}(\cdot)$ denoting either $\mu_{1,[\eta]}(\cdot)$ or $\mu_{\infty,[\eta]^{-1}}(\cdot)$
	the minimax problems
	\begin{align*}
	&\inf_{\eta \in {\realpositive^n}}\max_{d \in [\dmin,\dmax]^n} \mu_{[\eta]}([c] + [d]A), \\
	&\inf_{\eta \in {\realpositive^n}}\max_{d \in [\dmin,\dmax]^n} \mu_{[\eta]}([c] + A[d]),
	\end{align*}
	may each be solved by a bisection algorithm, each step of which is an LP. 
\end{corollary}
\begin{arxiv}
\begin{proof}
  The proof is an immediate consequence of the
  formulas~\eqref{fact:ms:1}-\eqref{fact:ms:4} as well as the fact that a
  $\max$ of quasiconvex functions is quasiconvex. Therefore, a bisection
  algorithm similar to the one in
  Theorem~\ref{thm:quasiconvexity}\ref{item:bisection} may be used to
  compute the optimal $\eta$.
\end{proof}
\end{arxiv}

\subsection{Monotonicity of diagonally-weighted log norms}

\begin{theorem}[Monotonicity of $\alpha$ and $\mu$]\label{thm:MetzlerHurwitzness}
	For any $A\in\real^{n\times{n}}$
	\begin{enumerate}
		\item \label{item:monotone-alpha} $\ds\alpha(A)\leq\alpha(\metzler{A})$,
		\item \label{item:monotone-mu} for all $p\in[1,\infty]$ and
		$\eta\in\realp^n$, we have $\mu_{p,[\eta]}(A)\leq\mu_{p,[\eta]}(\metzler{A})$,
		with equality holding for $p \in \{1,\infty\}$.
		\item \label{item:optimaldiagonalweights} For $p \in \{1,\infty\}$,
		\begin{equation*}
		\hspace{-\leftmargin}
		\inf_{\eta \in \realpositive^n} \mu_{p,[\eta]}(A)
		= \alpha(\metzler{A}) \geq \alpha(A).
		\end{equation*}
	\end{enumerate}
\end{theorem}

	Theorem~\ref{thm:MetzlerHurwitzness}\ref{item:optimaldiagonalweights} demonstrates that using diagonally-weighted $\ell_1$ and $\ell_\infty$ log norms, the best bound one can achieve on
	$\alpha(A)$ is $\alpha(\metzler{A})$, which may be conservative. In the following
	example, we show that the $\ell_2$ norm does not have the same conservatism.
	Despite the conservatism, Theorem~\ref{thm:quasiconvexity} demonstrates that optimizing
	diagonal weights is computationally efficient, being an LP at every step of the bisection,
	while optimizing weights for the $\ell_2$ norm is a semidefinite program at every step, which is more computationally challenging than an LP of similar dimension.
\begin{example}
	The matrix $A_* =\begin{bmatrix} 1 & 1 \\ -1 & 1 \end{bmatrix}$ has
	eigenvalues $\{1+\imag,1-\imag\}$ whereas $\metzler{A_*}$ has eigenvalues
	$\{2,0\}$. Therefore, $\alpha(A_*)=1<2=\alpha(\metzler{A_*})$. Additionally,
	$(A_*+A_*^\top)/2=I_2\implies\mu_2(A_*)=1$ and $\mu_2(\metzler{A_*})=2$.
\end{example}

\newcommand{\mcI}{\mathcal{I}}

\subsection{Log norms of principal submatrices}
Given a matrix $A\in\real^n$ and a non-empty index set
$\mcI\subset\until{n}$, let $A_\mcI\in\real^{|\mcI|\times|\mcI|}$ denote
the \emph{principal submatrix} obtained by removing the rows and columns of
$A$ which are not in $\mcI$.  Next, given a non-empty
$\mcI\subset\until{n}$, define the \emph{zero-padding map}
$\map{\operatorname{pad}_\mcI}{\real^{|\mcI|}}{\real^{n}}$ as follows:
$\operatorname{pad}_\mcI(y)$ is obtained by inserting zeros among the
entries of $y$ corresponding to the indices in $\until{n}\setminus
\mcI$. For example, with $n=3$ and $\mcI=\{1,3\}$, we define
$\operatorname{pad}_{\{1,3\}}(y_1,y_2)=(y_1,0,y_2)$. Then it is easy to see
that given a norm $\|\cdot\|$ on $\real^n$ and non-empty $\mcI \subset
\until{n}$, the map
$\map{\norm{\cdot}{\mcI}}{\real^{|\mcI|}}{\realnonnegative}$ defined by
$\norm{y}{\mcI} = \norm{\operatorname{pad}_\mcI(y)}{}$ is a norm on
$\real^{|\mcI|}$.

\begin{lemma}[Norm and log norm of principal submatrices]\label{lemma:submatrices}

	Assume $\norm{\cdot}{}$ is monotonic, let $\mu$ and $\mu_\mcI$
	denote the log norms associated to $\norm{\cdot}{}$ and
	$\norm{\cdot}{\mcI}$ respectively. Any matrix
	$A\in\real^{n\times{n}}$ satisfies
	\begin{enumerate}
		\item\label{fact:submatrix-norm-bound} $\norm{A_\mcI}{\mcI}\leq\norm{A}{}$,
		\item\label{fact:submatrix-mu-bound} $\mu_{\mcI}(A_\mcI)\leq\mu(A)$,
		\item\label{fact:diagonalmu-totallyH} if $\mu(A)<0$, then $A \in \mcTH$.
	\end{enumerate}
\end{lemma}

\begin{corollary}\label{cor:pruning}
  Suppose $A \in \mcMH \subset \real^{n \times n}$. Then
  \begin{enumerate}
  	\item\label{item:totally-contracting} $A_{\mcI} \in \mcMH \subset \real^{|\mcI| \times |\mcI|}$
  	for every non-empty $\mcI \subset \until{n}$ and
  	\item\label{item:conn-contracting} $A - A_{ij}\vect{e}_{ij} \in \mcMH$ for all $i,j \in \until{n}, i \neq j$, where $\vect{e}_{ij}$ is a matrix with all zeros and unity in its $ij$-th entry.
  \end{enumerate}
\end{corollary}

In the context of Proposition~\ref{prop:preview}, since our sufficient condition for the contractivity of the Hopfield and firing-rate NNs is $-I_n + A \in \mcMH$, Corollary~\ref{cor:pruning} implies that this sufficient condition implies total and structural contractivity, i.e., the removal of any neuron or edge from the neural network yields a new neural network that remains contracting.

\section{One-sided Lipschitz maps and nonsmooth contraction theory}\label{sec:contractiontheory}

\subsection{Review of one-sided Lipschitz maps}

We review weak pairings and one-sided Lipschitz maps as introduced
in~\cite{AD-SJ-FB:20o}; see also the earlier works~\cite{GS:06,ZA-EDS:14b}.

\begin{definition}[Weak pairing]
  A \emph{weak pairing} on $\real^n$ is a map
  $\map{\WeakP{\cdot}{\cdot}}{\real^n \times \real^n}{\real}$ satisfying:
  \begin{enumerate}
  \item (Subadditivity and continuity in its first argument) $\WeakP{x_1 +
    x_2}{y} \leq \WeakP{x_1}{y} + \WeakP{x_2}{y},$ for all $x_1,x_2,y \in
    \real^n$ and $\WeakP{\cdot}{\cdot}$ is continuous in its first
    argument,
  \item (Weak homogeneity) $\WeakP{\alpha x}{y} = \WeakP{x}{\alpha y} =
    \alpha\WeakP{x}{y}$ and $\WeakP{-x}{-y} = \WeakP{x}{y}$ for all $x,y
    \in \real^n$, $\alpha \geq 0$,
  \item (Positive definiteness) $\WeakP{x}{x} > 0$ for all $x \neq
    \vectorzeros[n]$,
  \item (Cauchy-Schwarz) $|\WeakP{x}{y}| \leq
    \WeakP{x}{x}^{1/2}\WeakP{y}{y}^{1/2}$ for all $x,y \in \real^n$.
  \end{enumerate}
  Additionally, we say a weak pairing satisfies \emph{Deimling's
    inequality} if $\ds\WeakP{x}{y} \leq \|y\|\lim_{h\to0^+} h^{-1}(\|y +
    hx\| - \|y\|)$ for all $x,y \in \real^n$, where $\|\cdot\| =
  \WeakP{\cdot}{\cdot}^{1/2}$.
\end{definition}
Deimling's inequality is well-defined since $\WeakP{\cdot}{\cdot}^{1/2}$
defines a norm on $\real^n$. Conversely, if $\real^n$ is equipped with a
norm $\|\cdot\|$ then there exists a (possibly non-unique) weak pairing
$\WeakP{\cdot}{\cdot}$ such that $\|\cdot\| = \WeakP{\cdot}{\cdot}^{1/2}$;
see~\cite[Theorem~16]{AD-SJ-FB:20o}. Henceforth, we assume that
weak pairings satisfy Deimling's inequality.

The relationship between weak pairings and log norms is given by Lumer's equality.

\begin{lemma}[Lumer's equality~{\cite[Theorem~18]{AD-SJ-FB:20o}}]\label{lemma:Lumer}
	Let $\|\cdot\|$ be a norm on $\real^n$ with compatible weak pairing
        $\WeakP{\cdot}{\cdot}$.  Then
	\begin{equation}
	\mu(A) = \sup_{x \in \real^n, x \neq \vectorzeros[n]} \frac{\WeakP{Ax}{x}}{\|x\|^2}, \qquad \text{for all } A \in \real^{n\times n}.
	\end{equation}
\end{lemma}

\begin{definition}[One-sided Lipschitz maps~{\cite[Definition~26]{AD-SJ-FB:20o}}]
	Consider $\map{f}{U}{\real^n}$ where $U \subseteq \real^n$ is open and connected.
	We say $f$ is \emph{one-sided Lipschitz} with respect to a weak pairing $\WeakP{\cdot}{\cdot}$
	if there exists $b \in \real$ such that
	\begin{equation*}
	\WeakP{f(x) - f(y)}{x - y} \leq b\|x - y\|^2, \quad \text{ for all } x,y \in U.
	\end{equation*}
	We say $b$ is a \emph{one-sided Lipschitz constant of $f$}. Moreover, the minimal
	one-sided Lipschitz constant of $f$ is
	\begin{equation}
	\osL(f) := \sup_{x,y \in U, x \neq y} \frac{\WeakP{f(x) - f(y)}{x - y}}{\|x - y\|^2}.
	\end{equation}
\end{definition}
If $f$ is continuously differentiable and $U$ is convex, it can be shown that
$\osL(f) = \sup_{x \in U} \mu(\jac{f}(x))$, where $\jac{f} := \frac{\partial f}{\partial x}$ is the Jacobian matrix of $f$.

A vector field $\map{f}{\real^n}{\real^n}$ satisfying $\osL(f) \leq -c < 0$
is said to be \emph{strongly infinitesimally contracting} with rate
$c$. A consequence of $\osL(f) \leq -c < 0$ is that the function $V(x,y) = \|x-y\|$
serves as an incremental Lyapunov function~\cite{DA:02} establishing incremental exponential stability, i.e.,
any two trajectories $x(\cdot), y(\cdot)$ satisfying $\dot{x} = f(x)$ and $\dot{y} = f(y)$
additionally satisfy $\|x(t) - y(t)\| \leq e^{-ct}\|x(0) - y(0)\|$ for all
$t \geq 0$~\cite[Theorem~31]{AD-SJ-FB:20o}. Moreover, if $f$ is continuous, then all solutions converge to
a unique equilibrium. Thus, in order to establish contractivity of neural network dynamics,
it is equivalent to establish that their one-sided Lipschitz constants are negative.

\subsection{Nonsmooth contraction theory}

In this section we consider locally Lipschitz $f$ and show that in this case, the
definition of $\osL$ does not depend on the weak pairing and instead depends only
on the norm through the log norm.

\begin{theorem}[$\osL$ simplification for locally Lipschitz $f$]\label{theorem:osLequivalence}
  For $\map{f}{U}{\real^n}$ locally Lipschitz on an open convex set, $U
  \subseteq \real^n$.
Then for every $c \in \real$ the following statements
  are equivalent:
  \begin{enumerate}
  \item\label{item:innerprodosL} $\ds \osL(f) \leq c$,
  \item\label{item:nonsmoothosL} $\ds\mu(\jac{f}(x))
    \leq c$ for almost every $x \in U$.
  \end{enumerate}
Specifically, $\osL(f) = \esssup_{x \in U} \mu(\jac{f}(x))$, where $\esssup$ denotes the essential supremum.
\end{theorem}
\begin{arxiv}
Recall that $\jac{f}(x)$ exists for almost every $x \in U$ by Rademacher's
theorem and thus the essential supremum ignores the Lebesgue measure zero set where $\jac{f}$ doesn't exist.

Theorem~\ref{theorem:osLequivalence} demonstrates that locally Lipschitz
$f$ enjoy a similar simplification in the $\osL$ definition as do
continuously differentiable functions.

In neural network models, nonsmooth activation functions such as ReLU,
LeakyReLU, and nonsmooth saturation functions are prevalent;
Theorem~\ref{theorem:osLequivalence} allows us to use standard log norm
results to analyze these models. In other words, for a given
continuous-time neural network dynamics $\dot{x} = \subscr{f}{NN}(x)$, with
locally Lipschitz $\subscr{f}{NN}$, to establish contractivity, it suffices to verify
that $\mu(\jac{\subscr{f}{NN}}(x)) \leq -c$ for almost every $x$.

\section{Contracting neural network dynamics}\label{sec:neuraldynamics}
In this section, we prove Proposition~\ref{prop:preview} in greater generality. 
Specifically, we establish tight estimates for the one-sided Lipschitz constant for both the Hopfield and firing-rate NNs with respect to both diagonally-weighted $\ell_1$ and $\ell_\infty$ norms. In instances where the one-sided Lipschitz constant is negative, we conclude that the neural network is strongly infinitesimally contracting. Beyond the Hopfield and firing-rate NNs, we also study the non-Euclidean contractivity of other classes of NNs.

\subsection{One-sided Lipschitz characterization of Hopfield NNs}
Recall the Hopfield NN dynamics, 
first introduced in~\cite{JJH:84}:
\begin{equation}\label{eq:neural-ode}
  \dot{x}=-C x+A\Phi(x)+u =: \fHNN(x),
\end{equation}
where $C \in \real^{n \times n}$ is a positive semi-definite diagonal matrix,
$A\in \real^{n \times n}$ is arbitrary, $u \in \real^n$ is a constant input, and $\map{\Phi}{\real^n}{\real^n}$ is an activation function. We make the following assumption on our activation functions:
\begin{assumption}[Activation functions]\label{assumption:activations}
	Activation functions are locally Lipschitz and diagonal, i.e., $\Phi(x) = (\phi_1(x_1), \dots, \phi_n(x_n))$ where each $\map{\phi_i}{\real}{\real}$ satisfies
	\begin{equation}\label{eq:slope-restricted}
	\begin{aligned}
	\dmin & := \inf_{x,y\in \real, x\neq y}\frac{\phi_i(x) - \phi_i(y)}{x - y} > -\infty, \\
	\dmax & := \sup_{x,y\in \real, x\neq y}\frac{\phi_i(x) - \phi_i(y)}{x - y}.
	\end{aligned}
	\end{equation}
	When $\phi_i$ satisfies~\eqref{eq:slope-restricted} with finite $\dmax$, we write $\phi_i \in \mathrm{slope}[\dmin,\dmax]$. If $\dmax = \infty$, we write $\phi_i \in \slope{\dmin}{\infty}$.
\end{assumption}

Assumption~\ref{assumption:activations} with finite $\dmax$ implies that activation functions are (globally) Lipschitz and that $\phi_i'(x) \in [\dmin,\dmax]$ for almost every $x \in \real$.
Many common activation functions satisfy these assumptions including ReLU,
$\tanh$ and sigmoids. If, instead, $\dmax = \infty$, we can consider locally Lipschitz activation functions with unbounded slope including rectified polynomials $\phi(x) = \max\{x,0\}^r$ for $r \in \mathbb{Z}_{\geq 0}$ which have been studied in~\cite{DK-JJH:16}.
Note that compared to Assumption~\ref{assm-simplified}, our activation functions do not need to be differentiable and are permitted more arbitrary bounds on their slopes.

The following theorem is the counterpart to Proposition~\ref{prop:preview}\ref{item:prop11} under more general assumptions.

\begin{theorem}[$\ell_1$ one-sided Lipschitzness of Hopfield neural network]
  \label{thm:osL-neural}
  Consider the Hopfield neural network~\eqref{eq:neural-ode} with each $\phi_i \in \slope{\dmin}{\dmax}$. Then
  \begin{enumerate}
  \item\label{osL:Hopf:1} for arbitrary
  $\eta\in\realp^n$, $\ds \osL_{1,[\eta]}(\fHNN)= \max\big\{ \mu_{1,[\eta]} (-C+\dmin
    A), \mu_{1,[\eta]} (-C+\dmax A)\big\}$.
  \item\label{osL:Hopf:2} the vector $\eta$ minimizing
    $\osL_{1,[\eta]}(\fHNN)$ is the solution to
    \begin{equation*}
    \begin{aligned}
    \hspace{-\leftmargin}
    \inf_{b \in \real, \eta \in \realpositive^n} & \qquad b \\
    \text{s.t.}\qquad  & (-C + \metzler{\dmin A}^\top) \eta \leq b \eta,\\
    & (-C + \metzler{\dmax A}^\top)\eta \leq b \eta,
    \end{aligned}
    \end{equation*}
	and if the infimum value is attained at parameter values $b^\star, \eta^\star$, then $\osL_{1,[\eta^\star]}(\fHNN) = b^\star$.
	\setcounter{saveenum}{\value{enumi}}
\end{enumerate}
	Further, suppose $\metzler{A}$ is irreducible. Then
	\begin{enumerate} \setcounter{enumi}{\value{saveenum}}
  \item\label{item:case1} if $C=cI_n$ and $\dmin \geq 0$, then, with $w_A\in\realp^n$ being
    the left dominant eigenvector of $\metzler{A}$,
    \begin{multline}\label{eq:case1}
    \hspace{-\leftmargin}
    \inf_{\eta \in \realpositive^n} \osL_{1,[\eta]}(\fHNN) =
      \osL_{1,[w_A]}(\fHNN) \\= -c + \max\{\dmin \alpha(\metzler{A}),
      \dmax\alpha(\metzler{A})\}.
  \end{multline}
  \item\label{item:case2} if $\dmin=0$ and $C\succ0$, then, with
    $w_*\in\realp^n$ being the left dominant eigenvector of $-C+\dmax
    \metzler{A}$,
    \begin{multline}
    \hspace{-\leftmargin}
    \inf_{\eta \in \realpositive^n} \osL_{1,[\eta]}(\fHNN) =
    \osL_{1,[w_*]}(\fHNN) \\= \max\big\{ \alpha(-C), \alpha(-C+\dmax\metzler{A}) \big\}.
    \end{multline}
  \end{enumerate}
\end{theorem}


In particular, Theorem~\ref{thm:osL-neural} provides \emph{exact values}
for the minimal one-sided Lipschitz constant of the Hopfield neural network
with respect to diagonally-weighted $\ell_1$ norms.

As a consequence of this theorem, suppose the $\inf$ in statement~\ref{osL:Hopf:2} is attained and let $b^\star, \eta^\star$ be
the optimal parameters for the LP. If $b^\star
< 0$, then the Hopfield neural network~\eqref{eq:neural-ode} is strongly
infinitesimally contracting with rate $|b^\star|$ with respect to
$\|\cdot\|_{1,[\eta^\star]}$. Note, in particular, that if $d_1 = 0, d_2 = 1$, $C = I_n$, and $\alpha(\metzler{A})<1$, statement~\ref{item:case1} is equivalent to the statement in Proposition~\ref{prop:preview}\ref{item:prop11}.



\begin{arxiv}
\begin{proof}[Proof of Theorem~\ref{thm:osL-neural}]
  Regarding statement~\ref{osL:Hopf:1}, for any $\eta\in\realp^n$,
  \begin{align*}
    &\osL_{1,[\eta]}(\fHNN) = \sup_{x\in\real^n \setminus \Omega_{f_H}}\mu_{1,[\eta]}(\jac{\fHNN}(x)) \\
    &= \sup_{x\in\real^n \setminus \Omega_{f_H}}\mu_{1,[\eta]}(-C+A\,\jac{\Phi}(x)) \\
    &= \max_{d\in[\dmin,\dmax]^n} \mu_{1,[\eta]}(-C+A[d])
    \\
    &= \max\big\{ \mu_{1,[\eta]} (-C+\dmin A), \mu_{1,[\eta]} (-C+\dmax A)\big\},
  \end{align*}
	where the second-to-last equality holds by Assumption~\ref{assumption:activations} and the last equality holds by Lemma~\ref{lemma:affine-scaling}.
	
	Statement~\ref{osL:Hopf:2} holds by Corollary~\ref{cor:minimax}.
	Regarding statement~\ref{item:case1},
  if $C=c I_n$ and $\dmin \geq 0$, then
  \begin{align*}
    \osL_{1,[\eta]}(\fHNN)
    &= -c +\max\big\{ \mu_{1,[\eta]} (\dmin A), \mu_{1,[\eta]} (\dmax A)\big\}
    \\
    &= -c +\max\big\{ \dmin \mu_{1,[\eta]} ( A), \dmax \mu_{1,[\eta]} (A)\big\}.
  \end{align*}
  Additionally, recall that $\eta=w_A$ is the optimal weight from
  Lemma~\ref{lemma:efficientlognorm-metzler} for the irreducible Metzler
  matrix $\metzler{A}$ with respect to $p=1$. Therefore,
  \begin{align*}
  &\inf_{\eta \in \realpositive^n} \osL_{1,[\eta]}(\fHNN) =
    \osL_{1,[w_A]}(\fHNN) \\&= -c + \max\{\dmin\mu_{1,[w_A]}(A), \dmax\mu_{1,[w_A]}(A)\}
    \\
    &= -c + \max\{\dmin\alpha(\metzler{A}), \dmax\alpha(\metzler{A})\}
  \end{align*}

  Regarding statement~\ref{item:case2}, if $\dmin=0$ and $C\succ0$, we compute
  \begin{align*}
    \osL_{1,[\eta]}(\fHNN)
    &= \max\big\{ \mu_{1,[\eta]} (-C), \mu_{1,[\eta]} (-C+\dmax A)\big\}
    \\
    &= \max\big\{ \alpha(-C), \mu_{1,[\eta]} (-C+\dmax A)\big\},
  \end{align*}
  which holds because $\mu_{1,[\eta]}(-C) = \max_{i\in\until{n}} -c_{ii} = \alpha(-C)$ for every $\eta \in \realpositive^n$. Additionally, we have that $\eta = w_*$ is the optimal weight for the irreducible Metzler matrix $-C + \dmax \metzler{A}$ by Lemma~\ref{lemma:efficientlognorm-metzler}. Thus,
  \begin{multline*}
  \inf_{\eta \in \realpositive^n} \osL_{1,[\eta]}(f_H) = \osL_{1,[w_*]}(f_H) \\
  = \max\{\alpha(-C), \alpha(-C + \dmax \metzler{A})\},
  \end{multline*}
  which proves the result.
\end{proof}

\begin{remark}[Comparison to~{\cite[Theorem~1]{LK-ME-JJES:22}, \cite[Exercise~2.22]{FB:23-CTDS}}]\label{rmk:SlotineMetzlerPlus}
	For $A \in \real^{n \times n}$, define its nonnegative Metzler majorant $\metzler{A}^+$ by
	$$(\metzler{A}^+)_{ij} = \begin{cases}
	\max\{A_{ii}, 0\}, \quad &\text{if } i = j, \\
	|A_{ij}|, &\text{if } i \neq j
	\end{cases}.$$
	In~\cite[Theorem~1]{LK-ME-JJES:22}, for $\dmin = 0$ and $C = I_n$ it is shown that if $\alpha(-I_n + \dmax\metzler{A}^+) < 0$, then the Hopfield neural network~\eqref{eq:neural-ode} is contracting with respect to a diagonally-weighted $\ell_2$ norm which is given in Lemma~\ref{lemma:efficientlognorm-metzler}.
	Compared to the condition $\alpha(-I_n + \metzler{A}^+) < 0$, the condition in Theorem~\ref{thm:osL-neural}\ref{item:case2} replaces $\metzler{A}^+$ with $\metzler{A}$ and thus guarantees that a larger class of synaptic matrices still guarantee contractivity of the Hopfield neural network.
\end{remark}

\end{arxiv}

Additionally, beyond Proposition~\ref{prop:preview}\ref{item:prop11}, we characterize the $\ell_\infty$ one-sided Lipschitz constant of the Hopfield NN in the following theorem.

\begin{theorem}[$\ell_\infty$ one-sided Lipschitzness of Hopfield neural network]
	\label{thm:osL-neural-inf}
	Consider the Hopfield neural network~\eqref{eq:neural-ode} with each $\phi_i \in \slope{\dmin}{\dmax}$. Let $\od = \max\{|\dmin|,\dmax|\}$. Then
	\begin{enumerate}
		\item\label{osL:Hopf-inf:1} for arbitrary
		$\eta\in\realp^n$, $\ds \osL_{\infty,[\eta]^{-1}}(\fHNN) = \max\{\mu_{\infty,[\eta]^{-1}}(-C + \od A - (\od - \dmin)(I_n \circ A)), \\ \mu_{\infty,[\eta]^{-1}}(-C + \od A - (\od - \dmax)(I_n \circ A))\}$.
		\item\label{osL:Hopf-inf:2} the vector $\eta$ minimizing
		$\osL_{\infty,[\eta]^{-1}}(\fHNN)$ is the solution to
		\begin{equation*}
		\begin{aligned}
		\hspace{-\leftmargin}
		\inf_{b \in \real, \eta \in \realpositive^n} & \qquad b \\
		\text{s.t.}\qquad  & (-C + \metzler{\od A - (\od - \dmin)(I_n \circ A)}) \eta \leq b \eta,\\
		& (-C + \metzler{\od A - (\od - \dmax)(I_n \circ A)})\eta \leq b \eta,
		\end{aligned}
		\end{equation*}
		and if the infimum value is attained at parameter values $b^\star, \eta^\star$, then $\osL_{\infty,[\eta^\star]^{-1}}(\fHNN) = b^\star$.
	\end{enumerate}
\end{theorem}
\begin{proof}
	Regarding statement~\ref{osL:Hopf-inf:1}, in analogy to the proof of Theorem~\ref{thm:osL-neural}\ref{osL:Hopf:1}, we have
	\begin{align*}
	&\osL_{\infty,[\eta]^{-1}}(\fHNN) = \max_{d \in [\dmin,\dmax]^n} \mu_{\infty,[\eta]^{-1}}(-C + A[d]) \\
	&= \max\{\mu_{\infty,[\eta]^{-1}}(-C + \od A - (\od - \dmin)(I_n \circ A)),\\ &\qquad \mu_{\infty,[\eta]^{-1}}(-C + \od A - (\od - \dmax)(I_n \circ A))\},
	\end{align*}
	where the final equality is by Lemma~\ref{lemma:affine-scaling}. Statement~\ref{osL:Hopf-inf:2} is then a consequence of Corollary~\ref{cor:minimax}.
\end{proof}


\subsection{One-sided Lipschitz characterization of firing-rate NNs}

Recall the firing-rate NN dynamics:
\begin{equation}\label{eq:firingrate}
\dot{x} = -Cx + \Phi(Ax + u) =: \fFR(x).
\end{equation}
The interpretation
for this name is that if $\Phi(x)$ is nonnegative for all $x \in \real^n$ (as is ReLU),
then the positive orthant is forward-invariant and $x$ is interpreted as
a vector of firing-rates, while in the Hopfield neural network, $x$ can be negative
and is thus interpreted as a vector of membrane potentials.


The following two theorems are generalizations of Proposition~\ref{prop:preview}\ref{item:prop12} under more general assumptions. Specifically, Theorem~\ref{thm:osL-firingrate} characterizes one-sided Lipschitzness of the firing-rate NN with respect to diagonally, weighted $\ell_\infty$ norms, while Theorem~\ref{thm:osL-fr-1} does the same with respect to diagonally-weighted $\ell_1$ norms.

\begin{theorem}[$\ell_\infty$ one-sided Lipschitzness of firing-rate neural network]
  \label{thm:osL-firingrate}
	Consider the firing-rate neural network~\eqref{eq:firingrate} with each $\phi_i \in \slope{\dmin}{\dmax}$ and invertible $A$.
	Then
	\begin{enumerate}
		\item\label{osL:FR:1} for arbitrary $\eta\in\realp^n$, $\osL_{\infty,[\eta]^{-1}}(\fFR)=\max\{ \mu_{\infty,[\eta]^{-1}}(-C+\dmin A), 
		\mu_{\infty,[\eta]^{-1}} (-C+\dmax A)\}$.
		\item\label{osL:FR:2} The choice of $\eta$ minimizing $\osL_{\infty,[\eta]^{-1}}(\fFR)$ is the solution to
		\begin{equation*}
		\begin{aligned}
		\hspace{-\leftmargin}
		\inf_{b \in \real, \eta \in \realpositive^n} & \qquad b \\
		\text{s.t.}\qquad  & (-C + \metzler{\dmin A})\eta \leq b\eta , \\
		& (-C + \metzler{\dmax A})\eta \leq b\eta,
		\end{aligned}
		\end{equation*}
		and if the infimum value is attained at parameter values $b^\star, \eta^\star$, then $\osL_{\infty,[\eta^\star]^{-1}}(\fFR) = b^\star$.
		\setcounter{saveenum}{\value{enumi}}
	\end{enumerate}
	Further, suppose that $\metzler{A}$ is irreducible. Then
	\begin{enumerate} \setcounter{enumi}{\value{saveenum}}
		\item\label{item:FRcase1} if $C=cI_n$ and $\dmin \geq 0$, then, with $v_A\in\realp^n$ being
		the right dominant eigenvector of $\metzler{A}$,
		\begin{multline}\label{eq:osL1}
		\hspace{-\leftmargin}
		\inf_{\eta \in \realpositive^n} \osL_{\infty,[\eta]}(\fFR) = \osL_{\infty,[v_A]^{-1}}(\fFR) \\ = -c + \max\{\dmin \alpha(\metzler{A}),
		\dmax\alpha(\metzler{A})\}.
		\end{multline}
		\item\label{item:FRcase2} if $\dmin=0$ and $C\succ0$, then, with
		$v_*\in\realp^n$ being the right dominant eigenvector of $-C+\dmax
		\metzler{A}$,
		\begin{multline}\label{eq:osL2}
		\hspace{-\leftmargin}
		\inf_{\eta \in \realpositive^n} \osL_{\infty,[\eta]}(\fFR) = \osL_{\infty,[v_*]^{-1}}(\fFR) \\= \max\big\{ \alpha(-C), \alpha(-C+\dmax\metzler{A}) \big\}.
		\end{multline}
	\end{enumerate}
\end{theorem}
\begin{arxiv}
\begin{proof}
	Regarding statement~\ref{osL:FR:1}, for any $\eta \in \realpositive^n$ we compute
	\begin{align*}
	&\osL_{\infty,[\eta]^{-1}}(\fFR) = \sup_{x\in\real^n \setminus \Omega_{\fFR}}\mu_{\infty,[\eta]^{-1}}(\jac{\fFR}(x)) \\
	&= \sup_{x\in\real^n \setminus \Omega_{\fFR}}\mu_{\infty,[\eta]^{-1}}(-C+\jac{\Phi}(Ax+u)A) \\
	&= \max_{d\in[\dmin,\dmax]^n} \mu_{\infty,[\eta]^{-1}}(-C+[d]A)
	\\
	&= \max\big\{ \mu_{\infty,[\eta]^{-1}} (-C+\dmin A), \mu_{\infty,[\eta]^{-1}} (-C+\dmax A)\big\},
	\end{align*}
	where the second-to-last equality holds 
	by Assumption~\ref{assumption:activations} and because $A$ is invertible. The last equality holds by Lemma~\ref{lemma:affine-scaling}.
	
	Statement~\ref{osL:FR:2} is a consequence of Corollary~\ref{cor:minimax}. Regarding statement~\ref{item:FRcase1}, if $C=c I_n$ and $\dmin \geq 0$, then
	\begin{align*}
	&\osL_{\infty,[\eta]^{-1}}(\fFR)
	\\&= -c +\max\big\{ \mu_{\infty,[\eta]^{-1}} (\dmin A), \mu_{\infty,[\eta]^{-1}} (\dmax A)\big\}
	\\
	&= -c +\max\big\{ \dmin \mu_{\infty,[\eta]^{-1}} ( A), \dmax \mu_{\infty,[\eta]^{-1}} (A)\big\}.
	\end{align*}
	Additionally, recall that $\eta=v_A$ is the optimal weight from
	Lemma~\ref{lemma:efficientlognorm-metzler} for the irreducible Metzler
	matrix $\metzler{A}$ with respect to $p=\infty$. Therefore,
	\begin{align*}
	&\inf_{\eta \in \realpositive^n} \osL_{\infty,[\eta]^{-1}}(\fFR) =
	\osL_{\infty,[v_A]^{-1}}(\fFR) \\&= -c + \max\{\dmin \mu_{\infty,[v_A]^{-1}}(A),\dmax
	\mu_{\infty,[v_A]^{-1}}(A)\}
	\\
	&=  -c + \max\{\dmin \alpha(\metzler{A}),\dmax\alpha(\metzler{A})\}.
	\end{align*}
	
	Regarding statement~\ref{item:case2}, if $\dmin=0$ and $C\succ0$, we compute
	\begin{align*}
	\osL_{\infty,[\eta]^{-1}}(\fFR)
	&= \max\big\{ \mu_{\infty,[\eta]^{-1}} (-C), \mu_{\infty,[\eta]^{-1}} (-C+\dmax A)\big\}
	\\
	&= \max\big\{ \alpha(-C), \mu_{\infty,[\eta]^{-1}} (-C+\dmax A)\big\},
	\end{align*}
	which holds because $\mu_{\infty,[\eta]^{-1}}(-C) = \max_{i\in\until{n}} -c_{ii} = \alpha(-C)$ for every $\eta \in \realpositive^n$. Additionally, we have that $\eta = v_*$ is the optimal weight for the irreducible Metzler matrix $-C + \dmax \metzler{A}$ by Lemma~\ref{lemma:efficientlognorm-metzler}. Thus,
	\begin{multline*}
	\inf_{\eta \in \realpositive^n} \osL_{\infty,[\eta]^{-1}}(\fFR) = \osL_{\infty,[v_*]^{-1}}(\fFR) \\
	= \max\{\alpha(-C), \alpha(-C + \dmax \metzler{A})\},
	\end{multline*}
	which proves the result.
\end{proof}

\begin{theorem}[$\ell_1$ one-sided Lipschitzness of firing-rate neural network]
	\label{thm:osL-fr-1}
	Consider the firing-rate neural network~\eqref{eq:firingrate} with each $\phi_i \in \slope{\dmin}{\dmax}$, and invertible $A$. 
Let $\od = \max\{|\dmin|,|\dmax|\}$. Then
	\begin{enumerate}
		\item\label{osL:fr-1:1} for arbitrary
		$\eta\in\realp^n$, $\ds \osL_{1,[\eta]}(\fFR) = \max\{\mu_{1,[\eta]}(-C + \od A - (\od - \dmin)(I_n \circ A)), \\ \mu_{1,[\eta]}(-C + \od A - (\od - \dmax)(I_n \circ A))\}$.
		\item\label{osL:fr-1:2} the vector $\eta$ minimizing
		$\osL_{1,[\eta]}(\fFR)$ is the solution to
		\begin{equation*}
		\begin{aligned}
		\hspace{-\leftmargin}
		\inf_{b \in \real, \eta \in \realpositive^n} & \qquad b \\
		\text{s.t.}\qquad  & (-C + \metzler{\od A - (\od - \dmin)(I_n \circ A)})^\top \eta \leq b \eta,\\
		& (-C + \metzler{\od A - (\od - \dmax)(I_n \circ A)})^\top\eta \leq b \eta,
		\end{aligned}
		\end{equation*}
		and if the infimum value is attained at parameter values $b^\star, \eta^\star$, then $\osL_{1,[\eta^\star]}(\fFR) = b^\star$.
	\end{enumerate}
\end{theorem}
\begin{proof}
	Regarding statement~\ref{osL:fr-1:1}, in analogy to the proof of Theorem~\ref{thm:osL-firingrate}\ref{osL:FR:1}, we have
	\begin{align*}
	&\osL_{1,[\eta]}(\fFR) = \max_{d \in [\dmin,\dmax]^n} \mu_{1,[\eta]}(-C + [d]A) \\
	&= \max\{\mu_{1,[\eta]}(-C + \od A - (\od - \dmin)(I_n \circ A)),\\ &\qquad \mu_{1,[\eta]}(-C + \od A - (\od - \dmax)(I_n \circ A))\},
	\end{align*}
	where the final equality is by Lemma~\ref{lemma:affine-scaling}. Statement~\ref{osL:fr-1:2} is then a consequence of Corollary~\ref{cor:minimax}.
\end{proof}

\begin{remark}
	For invertible $A \in \real^{n \times n}$,
	Theorems~\ref{thm:osL-firingrate}\ref{osL:FR:1} and~\ref{thm:osL-fr-1}\ref{osL:fr-1:1} provide an \emph{exact
		value} for the minimal one-sided Lipschitz constant of the firing rate
	model with respect to a given norm. If $A$ is not invertible, then the closure of the image of the map $x \mapsto \jac{\Phi}(Ax + u)$
	may not contain all the vertices of the set $[\dmin,\dmax]^n$. For non-invertible $A$ and arbitrary $\eta \in
	\realpositive^n$, the values presented in these theorems are instead upper bounds on the minimal one-sided Lipschitz constant.
\end{remark}

In Figure~\ref{fig:phase-portrait}, we plot the phase portrait of a $2$-dimensional firing-rate neural network,~\eqref{eq:firingrate}, along with level sets of the corresponding Lyapunov function. We highlight the utility of optimizing the weight of the $\ell_\infty$ norm. Namely, although the firing-rate neural network example is not contracting with respect to the $\ell_\infty$ norm, it is contracting with respect to a weighted $\ell_\infty$ norm, where the optimal diagonal weight is $[\eta^{\star}]^{-1}$, where $\eta^{\star}$ is the right dominant eigenvector of $\metzler{A}$.

\begin{figure}[t]
	\includegraphics[width=0.9\linewidth]{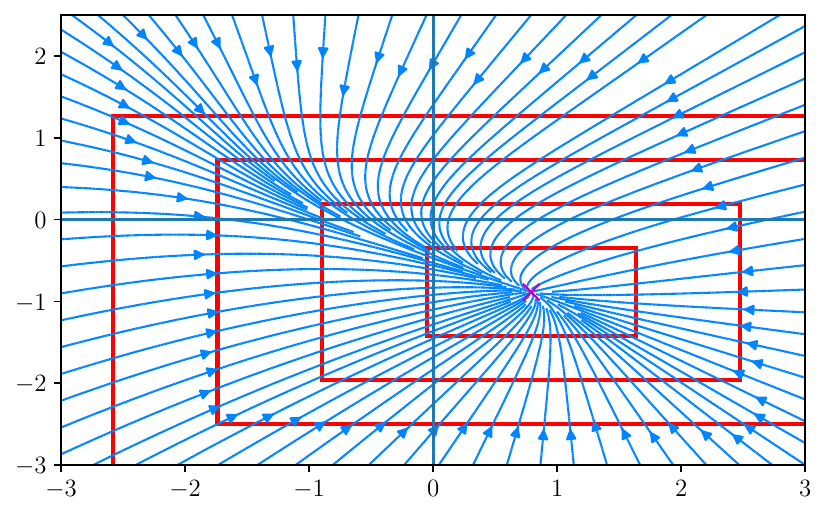}
	\caption{The phase portrait for a $2$-dimensional firing-rate neural network~\eqref{eq:firingrate} with $C = I_2$, $\Phi = \tanh$, $A = \left[\begin{smallmatrix} -0.1 & -1.3 \\ -0.4 & 0.1 \end{smallmatrix}\right]$, $u = (0,-1)$. The blue curves denote trajectories from varying initial conditions and the purple cross denotes the equilibrium point, $x^\star$. Following Theorem~\ref{thm:osL-neural-inf}\ref{item:FRcase2}, the right dominant eigenvector of $\metzler{A}$, $\eta \approx (1.57,1),$ yields a contraction rate of $1 - \alpha(\metzler{A}) \approx 0.272$ with respect to $\|\cdot\|_{\infty,[\eta]^{-1}}$. Note that the neural network is not contracting with respect to $\|\cdot\|_{\infty}$. Level sets of the Lyapunov function $V(x) = \|x-x^\star\|_{\infty,[\eta]^{-1}}$ are shown in red. We remark that this neural network is connectively contracting as described in Section~\ref{sec:preview}.}\label{fig:phase-portrait}
\end{figure}

\end{arxiv}
\subsection{Contractivity of Hopfield and firing-rate neural networks with unbounded slope}\label{subsection:unbounded}
In the spirit of the classic work~\cite{MF-AT:95} which studies Hopfield neural networks which have monotone activation functions with unbounded slope, we present the following result on the contractivity of Hopfield and firing-rate neural networks when $\phi_i \in \slope{d_1}{\infty}$.

\begin{theorem}[Contractivity under unbounded slope]\label{theorem:unbounded-slope}
	Consider the Hopfield neural network~\eqref{eq:neural-ode} and firing-rate neural network~\eqref{eq:firingrate} with $\phi_i \in \slope{d_1}{\infty}$ and irreducible $\metzler{A}$ with dominant left and right eigenvectors $w_A, v_A$, respectively and suppose that
	\begin{enumerate}[label=\textup{(A\arabic*)}]
		\item\label{assumption:MH} $A \in \mcMH$,
		\item\label{assumption:dissipation} $A \in \real^{n\times n}$, $C \succeq 0$, and $\dmin \in \real$ satisfy
		$$\hspace{-2em}\ds-\alpha(-C) + \max\{\dmin,0\}\alpha(\metzler{A})\! >\! -(|\dmin|-\dmin)\!\min_{i\in\until{n}}\!A_{ii}.
		$$
	\end{enumerate}
	Then
	\begin{enumerate}
		\item\label{item:unboundedHopfield} The Hopfield neural network~\eqref{eq:neural-ode} is strongly infinitesimally contracting with respect to $\|\cdot\|_{1,[w_A]}$ with rate $-\alpha(-C) + \max\{\dmin,0\}\alpha(\metzler{A}) + (|\dmin|-\dmin)\min_{i\in\until{n}}A_{ii} > 0$ and
		\item\label{item:unboundedFiring} The firing-rate neural network~\eqref{eq:firingrate} is strongly infinitesimally contracting with respect to $\|\cdot\|_{\infty,[v_A]^{-1}}$ with rate $-\alpha(-C) + \max\{\dmin,0\}\alpha(\metzler{A}) + (|\dmin|-\dmin)\min_{i\in\until{n}}A_{ii} > 0$.
	\end{enumerate}
\end{theorem}
\begin{proof}
	Regarding statement~\ref{item:unboundedHopfield}, we adopt the shorthand $r_i := A_{ii} + \sum_{j \neq i} |A_{ji}|(w_A)_j/(w_A)_i$. Then we observe that $A \in \mcMH$ implies that for every $i \in \until{n}$, $r_i \leq \alpha(\metzler{A})$ with $\alpha(\metzler{A}) < 0$. Then, for every $x \in \real^n \setminus \Omega_{\fHNN}$,
	\begin{align*}
	&\mu_{1,[w_{A}]}(\jac{\fHNN}(x)) = \mu_{1,[w_{A}]}(-C + A\jac{\Phi}(x)) \\
	&= \max_{i \in \until{n}} -c_i + A_{ii}\phi_i'(x_i) + \sum_{j\neq i} |A_{ji}\phi_i'(x_i)|\frac{({w_A})_j}{({w_A})_i} \\
	&= \max_{i \in \until{n}} -c_i + A_{ii}\phi_i'(x_i) + |\phi_i'(x_i)|\sum_{j\neq i} |A_{ji}|\frac{({w_A})_j}{({w_A})_i} \\
	&= \max_{i \in \until{n}} -c_i + |\phi_i'(x_i)|r_i - (|\phi_i'(x_i)| - \phi_i'(x_i))A_{ii} \\
	&\overset{\bigstar}{\leq} \max_{i \in \until{n}} -c_i + \max\{\dmin,0\}r_i - (|\dmin| - \dmin)A_{ii} \\
	&\leq \alpha(-C) + \max\{\dmin,0\}\alpha(\metzler{A}) - (|\dmin| - \dmin)\!\min_{i\in\until{n}} A_{ii}
	\end{align*}
	where inequality $\overset{\bigstar}{\leq}$ holds because $r_i < 0$ and $A_{ii} < 0$ for all $i$.
	Since this inequality holds for all $x \in \real^n \setminus \Omega_{\fHNN}$, we conclude that $\osL_{1,[w_A]}(\fHNN) \leq \alpha(-C) + \max\{\dmin,0\}\alpha(\metzler{A}) - (|\dmin| - \dmin)\min_{i\in\until{n}} A_{ii},$ which implies the result. The proof of statement~\ref{item:unboundedFiring} is essentially identical and thus omitted.
\end{proof}

\begin{remark}
	Note that in Theorem~\ref{theorem:unbounded-slope},
	\begin{enumerate}
		\item
		if $\dmin \geq 0$, then condition~\ref{assumption:MH} immediately implies condition~\ref{assumption:dissipation}. Hence, $A \in \mcMH$ is a sufficient condition for the strong infinitesimal contractivity of~\eqref{eq:neural-ode} and~\eqref{eq:firingrate} with unbounded-slope monotonic activation functions.
		\item alternatively if $\alpha(\metzler{A}) = 0$ and $C \succ 0$, then the condition $-\alpha(-C) > -(|\dmin|-\dmin)\min_{i\in\until{n}}A_{ii}$ is a sufficient condition for the strong infinitesimal contractivity of~\eqref{eq:neural-ode} and~\eqref{eq:firingrate}. Note that in this case, we only need $\metzler{A}$ to be marginally stable.
	\end{enumerate}
	
\end{remark}

\subsection{Contractivity of other continuous-time neural networks}\label{subsec:other-models}
We apply Theorem~\ref{thm:osL-neural} and the log norm results in
Section~\ref{sec:novel-mm} to the following related neural circuit
models, all of which are studied in the classic book~\cite{EK-AB:00}. In the following theorems, we assume all Metzler matrices are irreducible.

\begin{theorem}[Contractivity of special Hopfield models]\label{thm:special-Hopfield}
  \begin{enumerate}
  \item\label{item:Persidskii} If $A \in \mcMH$, and $\dmin > 0$, the Persidskii-type\footnote{See\cite[Definition~3.2.1]{EK-AB:00}}
    model $$\dot{x}=A\Phi(x)$$ 
    with each $\phi_i \in \slope{\dmin}{\dmax}$ is strongly infinitesimally contracting with respect to norm
    $\norm{\cdot}{1,[w_A]}$ with rate $\dmin|\alpha(\metzler{A})|$.
		
  \item\label{item:specialHopfield} If $-C+\dmax A \in \mcMH$, the Hopfield neural network~\eqref{eq:neural-ode} with
    $\dmin=0$ and positive diagonal $C$ is strongly infinitesimally
    contracting with respect to $\norm{\cdot}{1,[w_*]}$ with rate
    $-\max\big\{ \alpha(-C), \alpha(-C+\dmax\metzler{A}) \big\}>0$.
  \end{enumerate}
\end{theorem}
\begin{arxiv}

\begin{proof}
	Regarding statement~\ref{item:Persidskii}, let $\fPer(x):= A\Phi(x)$. By Theorem~\ref{thm:osL-neural}\ref{item:case1} with $c = 0$,
	$
	\osL_{1,[w_A]}(\fPer) = \max\{\dmin\alpha(\metzler{A}), \dmax\alpha(\metzler{A})\}.
	$
	However, since $A \in \mcMH$, $\alpha(\metzler{A}) < 0$, so $\osL_{1,[w_A]}(\fPer) = \dmin\alpha(\metzler{A})$. Thus, the Persidskii-type model is strongly infinitesimally contracting with respect to norm $\|\cdot\|_{1,[w_A]}$ with rate $\dmin|\alpha(\metzler{A})|$.
	
	Regarding statement~\ref{item:specialHopfield}, by Theorem~\ref{thm:osL-neural}\ref{item:case2},
	$$\osL_{1,[w_*]}(\fHNN) = \max\big\{ \alpha(-C), \alpha(-C+\dmax\metzler{A}) \big\}.$$
	In particular, since $-C + \dmax A \in \mcMH$, $\alpha(-C + \dmax\metzler{A}) < 0$ and since $C$ is positive diagonal, we have $\osL_{1,[w_*]}(\fHNN) < 0$ so that the Hopfield neural network is strongly infinitesimally contracting with respect to $\|\cdot\|_{1,[w_*]}$ with rate $-\max\big\{ \alpha(-C), \alpha(-C+\dmax\metzler{A}) \big\}>0$.
\end{proof}
\end{arxiv}

\begin{theorem} \label{thm:unnamed-model}
From \cite[Theorem~3.2.4]{EK-AB:00}, consider
$$\dot{x}=Ax-C\Phi(x),$$ with diagonal $C\succeq0$ and each $\phi_i \in \slope{\dmin}{\dmax}$.
If $A-\dmin C \in \mcMH$ with corresponding dominant left eigenvector
$w_{**}$, then this model is strongly infinitesimally contracting with
respect to $\|\cdot\|_{1,[w_{**}]}$ with rate $-\alpha(\metzler{A} - \dmin
  C) > 0$.
\end{theorem}
\begin{arxiv}
\begin{proof}
	We compute the one-sided Lipschitz constant of $f(x):= Ax - C\Phi(x)$ with respect to norm $\|\cdot\|_{1,[w_{**}]}$.
	\begin{align*}
	&\osL_{1,[w_{**}]}(f) = \sup_{x\in\real^n \setminus \Omega_{f}} \mu_{1,[w_{**}]}(\jac{f}(x))\\
	&= \sup_{x\in\real^n \setminus \Omega_{f}} \mu_{1,[w_{**}]}(A - C\jac{\Phi}(x)) \\
	&\overset{\bigstar}{=} \max_{d \in [\dmin,\dmax]^n} \mu_{1,[w_{**}]}(A - C[d]) \overset{\spadesuit}{=} \mu_{1,[w_{**}]}(A - \dmin C) \\
	&\overset{\clubsuit}{=} \mu_{1,[w_{**}]}(\metzler{A - \dmin C}) \overset{\blacklozenge}{=} \alpha(\metzler{A} - \dmin C).
	\end{align*}
	where the equality $\overset{\bigstar}{=}$ is due to Assumption~\ref{assumption:activations}, equality $\overset{\spadesuit}{=}$ is because $C \succeq 0$, equality $\overset{\clubsuit}{=}$ is by Theorem~\ref{thm:MetzlerHurwitzness}\ref{item:monotone-mu}, and $\overset{\blacklozenge}{=}$ is by Lemma~\ref{lemma:efficientlognorm-metzler}. Moreover, since $A - \dmin C \in \mcMH$, $\alpha(\metzler{A} - \dmin C) < 0$ so $f$ is strongly infinitesimally contracting with respect to $\|\cdot\|_{1,[w_{**}]}$ with rate $-\alpha(\metzler{A} - \dmin C) > 0$.
\end{proof}	

\end{arxiv}

\begin{theorem}\label{thm:Hadamar-model}
  From \cite[Theorem~3.2.10]{EK-AB:00}, consider $$\dot{x}_i=
  \sum\nolimits_{j=1}^n A_{ij}\phi_{ij}(x_j)$$
  for each $i \in \until{n}$
  and with each $\phi_{ij} \in \slope{\dmin}{\dmax}$. If $\dmin > 0$ and $$B := \dmax A - (\dmax - \dmin)(I_n
  \circ A) \in \mcMH,$$ with corresponding dominant left and right
  eigenvectors $w_B, v_B$, respectively, then this model is strongly
  infinitesimally contracting with rate $-\alpha(\metzler{B}) > 0$ with respect
  to both $\|\cdot\|_{1,[w_B]}$ and $\|\cdot\|_{\infty,[v_B]^{-1}}$.
\end{theorem}
\begin{arxiv}
\begin{proof}
	First note that the assumption $B \in \mcMH$ implies that $A_{ii} < 0$ for every $i \in \until{n}$ since the diagonal elements of $B$ are $\dmin A_{ii}$ and a necessary condition for $B \in \mcMH$ is $B_{ii} < 0$ since $\mcMH \subset \mcTH$. Let $f$ denote the vector field given by $f_i(x) = \sum_{j=1}^n A_{ij}\phi_{ij}(x_j)$. We compute
	$(\jac{f}(x))_{ij} = \frac{\partial}{\partial x_j} \sum_{j=1}^n A_{ij}\phi_{ij}(x_j) = A_{ij}\phi_{ij}'(x_j)$
	for almost every $x \in \real^n$. In other words,
	$\jac{f}(x) = A \circ \jac{\Phi}(x),$
	for almost every $x \in \real^n$, where $(\jac{\Phi}(x))_{ij} = \phi_{ij}'(x_j)$. We now proceed to elementwise upper bound $\metzler{\jac{f}(x)}$. Observe that for every $i \neq j \in \until{n}$,
	\begin{align*}
	(\metzler{\jac{f}(x)})_{ij} &= |A_{ij}\phi_{ij}'(x_j)| \leq \dmax |A_{ij}| = (\metzler{B})_{ij}, \\
	(\metzler{\jac{f}(x)})_{ii} &= A_{ii}\phi_{ii}'(x_i) \leq \dmin A_{ii} = (\metzler{B})_{ii},
	\end{align*}
	where the second inequality holds because $A_{ii} < 0$ for every $i \in \until{n}$. Now observe that for any matrix $A \in \real^{n \times n}$, if $\metzler{A} \leq A'$ elementwise, then both $\mu_{1,[\eta]}(A) \leq \mu_{1,[\eta]}(A')$ and $\mu_{\infty,[\eta]^{-1}}(A) \leq \mu_{\infty,[\eta]^{-1}}(A')$ hold for any $\eta \in \realpositive^n$. Then we can observe that
	\begin{align*}
	&\osL_{1,[w_B]}(f) = \sup_{x\in\real^n \setminus \Omega_{f}} \mu_{1,[w_B]}(\jac{f}(x)) \\
	&= \sup_{x\in\real^n \setminus \Omega_{f}} \mu_{1,[w_B]}(\metzler{\jac{f}(x)})
	\leq \sup_{x\in\real^n \setminus \Omega_{f}} \mu_{1,[w_B]}(\metzler{B}) \\
	&= \mu_{1,[w_B]}(\metzler{B}) = \alpha(\metzler{B}),
	\end{align*}
	where the final equality holds by Lemma~\ref{lemma:efficientlognorm-metzler}. An analogous computation shows that $\osL_{\infty,[v_B]^{-1}}(f) \leq \alpha(\metzler{B})$. As a consequence, since $B \in \mcMH$, this model is strongly infinitesimally contracting with respect to both $\|\cdot\|_{1,[w_B]}$ and $\|\cdot\|_{\infty,[v_B]^{-1}}$ with rate $-\alpha(\metzler{B}) > 0.$
\end{proof}
\end{arxiv}

The next two theorems serve as non-Euclidean versions of early results on
contractivity of Lur'e systems (in application to the entrainment problem)
established first in~\cite{VAY:64}.

\begin{theorem}[Contractivity of Lur'e system]
  \label{thm:Lure-model}
	From~\cite[Theorem~3.2.7]{EK-AB:00}, consider the Lur'e system
	\begin{align*}
	\dot{x} &= Ax + v\phi(y), \\
	y &= w^\top x,
	\end{align*}
	where $A \in \real^{n \times n}, v,w \in \real^n$ and $\phi \in \slope{\dmin}{\dmax}$. Consider the
	following two infimization problems:
	\begin{equation}\label{eq:lurel1}
	\begin{aligned}
	\hspace{-\leftmargin}
	\inf_{b \in \real, \eta \in \realpositive^n} & \qquad b \\
	\text{s.t.}\qquad  & \metzler{A + \dmin vw^\top}^\top \, \eta \leq b \eta,\\
	& \metzler{A + \dmax vw^\top}^\top\, \eta \leq b \eta,
	\end{aligned}
	\end{equation}
        and
	\begin{equation}\label{eq:lurelinf}
	\begin{aligned}
	\hspace{-\leftmargin}
	\inf_{c \in \real, \xi \in \realpositive^n} & \qquad c \\
	\text{s.t.}\qquad  & \metzler{A + \dmin vw^\top} \, \xi \leq c \xi,\\
	& \metzler{A + \dmax vw^\top}\, \xi \leq c \xi.
	\end{aligned}
	\end{equation}
	 Let $b^\star, c^\star$ be infimum values for~\eqref{eq:lurel1},~\eqref{eq:lurelinf}, respectively. Then
	\begin{enumerate}
	\item\label{item:Lure1} if $b^\star < 0$, then for every $\varepsilon \in {]0,|b^\star|[}$, there exists $\eta \in \realpositive^n$ such that the closed-loop dynamics are strongly
          infinitesimally contracting with rate $|b^\star| - \varepsilon > 0$ with respect to
          $\|\cdot\|_{1,[\eta]}$.
	\item\label{item:Lure2} if $c^\star < 0$, then for every $\varepsilon \in {]0,|c^\star|[}$, there exists $\xi \in \realpositive^n$ such that the closed-loop dynamics are strongly
          infinitesimally contracting with rate $|c^\star| - \varepsilon > 0$ with respect to
          $\|\cdot\|_{\infty,[\xi]^{-1}}$.
	\end{enumerate}
	
\end{theorem}
\begin{arxiv}
\newcommand{\fLure}{\subscr{f}{L}}
\begin{proof}
	Let $\fLure(x) := Ax + v\phi(w^\top x)$. Regarding statement~\ref{item:Lure1}, computing the one-sided Lipschitz constant of $\fLure$ with respect to $\|\cdot\|_{1,[\eta]}$ for arbitrary $\eta \in \realpositive^n$ yields
	\begin{align*}
	&\osL_{1,[\eta]}(\fLure) = \sup_{x\in\real^n \setminus \Omega_{\fLure}} \mu_{1,[\eta]}(\jac{\fLure}(x)) \\
	&= \sup_{x\in\real^n \setminus \Omega_{\fLure}} \mu_{1,[\eta]}(A + v\phi'(w^\top x)w^\top) \\
	&\overset{\bigstar}{=} \max_{d \in [\dmin,\dmax]} \mu_{1,[\eta]}(A + d\;vw^\top)\\
	&\overset{\spadesuit}{=} \max\{\mu_{1,[\eta]}(A + \dmin vw^\top), \mu_{1,[\eta]}(A + \dmax vw^\top)\},
	\end{align*}
	where $\overset{\bigstar}{=}$ holds by Assumption~\ref{assumption:activations} on $\phi$ and $\overset{\spadesuit}{=}$ holds because the maximum of a convex function ($\mu$ in this case) over a compact interval occurs at one of the endpoints of the interval. As a consequence, $\osL_{1,[\eta]}(\fLure) < 0$ if and only if
	$$\inf_{\eta \in \realpositive^n}\max\{\mu_{1,[\eta]}(A + \dmin vw^\top), \mu_{1,[\eta]}(A + \dmax vw^\top)\} < 0.$$
	Therefore, if $b^\star < 0$ for problem~\eqref{eq:lurel1}, then, by a continuity argument, for every $\varepsilon \in {]0,|b^\star|[}$, there exists $\eta \in \realpositive^n$ such that $\mu_{1,[\eta]}(A + \dmin vw^\top) \leq b^\star +\varepsilon$ and $\mu_{1,[\eta]}(A + \dmax vw^\top) \leq b^\star+\varepsilon$. Therefore, if $b^\star < 0$, then we conclude that the Lur'e system is strongly infinitesimally contracting with respect to $\|\cdot\|_{1,[\eta]}$ with rate $|b^\star|-\varepsilon$. The proof of statement~\ref{item:Lure2} is essentially identical, replacing $\|\cdot\|_{1,[\eta]}$ with $\|\cdot\|_{\infty,[\xi]^{-1}}$.
\end{proof}
\end{arxiv}

\begin{theorem}[Multivariable Lur'e system]
  \label{thm:Lure-model-multi}
  Consider the multivariable Lur'e system
  \begin{equation}\label{eq:multi-Lure}
  \begin{aligned}
  \dot{x} &= Ax + B\Phi(y), \\
  y &= Cx,
  \end{aligned}
  \end{equation}
  where $A \in \real^{n \times n}, B \in \real^{n \times m}, C \in \real^{m
    \times n}$, $\phi_i \in \slope{\dmin}{\dmax}$ with $\dmin \geq 0$ for all $i \in \until{m}$. Define $(\cdot)_+$ and $(\cdot)_{-}$
  by $(x)_+ = \max\{x,0\}$ and $(x)_{-} = \min\{x,0\}$. Define $F \in
  \real^{n \times n}$ componentwise by
  \begin{equation*}
    \begin{aligned}
      F_{ii} &= A_{ii} + \dmax \sum_{j=1}^m (B_{ij}C_{ji})_+ + \dmin \sum_{j=1}^m (B_{ij}C_{ji})_-, \\
      F_{ij} &= |A_{ij}| +
      \max \left\{ \dmax \sum_{k=1}^m (B_{ik}C_{kj})_+ + \dmin \sum_{k=1}^m (B_{ik}C_{kj})_{-}, \right.\\
      &\qquad \quad \left.
      -\dmin \sum_{k=1}^m (B_{ik}C_{kj})_+ - \dmax \sum_{k=1}^m (B_{ik}C_{kj})_- \right\},
    \end{aligned}
  \end{equation*}
  for $i \neq j$. Then, if $F \in \mcMH$ with corresponding dominant left
  and right eigenvectors $w_F, v_F$, the closed-loop dynamics are strongly
  infinitesimally contracting with rate $-\alpha(\metzler{F})>0$ with respect
  to both $\|\cdot\|_{1,[w_F]}$ and $\|\cdot\|_{\infty,[v_F]^{-1}}$.
\end{theorem}
\begin{arxiv}
\begin{proof}
  Let $\fMLure(x)=Ax+B\Phi(Cx)$ and note $\jac{\fMLure}(x)=A+B[d]C$ for
  some $d\in[\dmin,\dmax]^n$. Also note
  $(B[d]C)_{ij}=\sum_{k=1}^mB_{ik}d_kC_{kj}$. The proof follows from noting
  that the matrix $F$ is an entry-wise upper bound on $\metzler{\jac{\fMLure}(x)}$, for all $x$, in analogy with the proof of Theorem~\ref{thm:Hadamar-model}.
\end{proof}
\end{arxiv}

Finally, we present a sharper condition for the non-Euclidean contractivity of the multivariable Lur'e system with $\dmin$ that can be negative. For $\eta \in \realpositive^n, \od = \max\{|\dmin|,|\dmax|\}, M:= |A| + \od|B||C|$, and $g = \|A\|_{\infty,[\eta]^{-1}} + \od\|BC\|_{\infty,[\eta]^{-1}}$, consider the following mixed-integer linear program (MILP):
\begin{align}
&\max_{y \in \real,Z \in \real^{n\times n}, d \in [\dmin,\dmax]^m, W \in \{0,1\}^{n \times n}} \quad y, \nonumber\\
& \qquad\qquad \text{subject to} \nonumber\\
& Z \leq A + B[d]C + 2M \circ (W - (I_n \circ W)), \label{eq:MILP}\\
& Z \leq -A - B[d]C + 2M \circ (\vectorones[n]\vectorones[n]^\top - (W - (I_n \circ W))), \nonumber\\
& y \leq (A + B[d]C)_{ii} + \sum_{j\neq i} Z_{ij}\frac{\eta_j}{\eta_i} + 2 g W_{ii},
\quad \forall i \in \until{n}, \nonumber\\
& \mathrm{Trace}(W) = n - 1. \nonumber
\end{align}

\begin{theorem}[One-sided Lipschitzness of multi-variable Lur'e system]\label{theorem:MLure-MILP}
	Consider the multi-variable Lur'e system~\eqref{eq:multi-Lure}, let $\fMLure(x) = Ax + B\Phi(Cx)$ be the closed-loop dynamics with each $\phi_i \in \slope{\dmin}{\dmax}$ and let $y^\star$ be the optimal value for the MILP~\eqref{eq:MILP}. Then the following statements hold
	\begin{enumerate}
		\item $\osL_{\infty,[\eta]^{-1}}(\fMLure) \leq y^\star$.
		
		\item If $C$ is full row rank, then $\osL_{\infty,[\eta]^{-1}}(\fMLure) = y^\star$.
	\end{enumerate}
\end{theorem}

\begin{proof}
	Note that $$\osL_{\infty,[\eta]^{-1}}(\fMLure) \leq \max\nolimits_{d \in [\dmin,\dmax]^m}\mu_{\infty,[\eta]^{-1}}(A + B[d]C),$$
	with equality holding if $C$ is full row rank. Therefore, all that remains is to show that $\max_{d \in [\dmin,\dmax]^m}\mu_{\infty,[\eta]^{-1}}(A + B[d]C) = y^\star$. The proof of this result is a consequence of the formula for $\mu_{\infty,[\eta]^{-1}}$, using a so-called ``big-$M$" formulation (see, e.g.,~\cite[Section~III.C.]{AR-JH:05-paper}) with $Z_{ij} \leq |(A + B[d]C)_{ij}|$ for $i \neq j$ and $y \leq \max_{i \in \until{n}}(A + B[d]C)_{ii} + \sum_{j\neq i}Z_{ij}$.
\end{proof}
The challenge of additionally optimizing $\eta \in \realpositive^n$ so that $\osL_{\infty,[\eta]^{-1}}(\fMLure)$ is minimized remains an open problem.

\section{Discussion}

In this paper, we present novel non-Euclidean log norm results and a
non-smooth contraction theory simplification and we apply these results to
study the contractivity of continuous-time NN models, primarily focusing on the Hopfield
and firing-rate models. We provide efficient algorithms for computing
optimal non-Euclidean contraction rates and corresponding norms. Our approach
is robust with respect to activation function and additional unmodeled
dynamics and, more generally, establishes the strong contractivity property
which, in turn, implies strong robustness properties.

As a first direction of future research, we plan to investigate
the multistability of continuous-time neural networks via generalizations of contraction theory. 
Contraction theory
ensures the uniqueness of a globally exponentially stable equilibrium, but
several classes of neural networks exhibit multiple equilibria~\cite{CYC-KHL-CWS:06}.
As a second direction, we plan to investigate the role of non-Euclidean contractivity
in neural networks for controller design and system identification in the spirit of the
works~\cite{YW:17,HY-PS-MA:22,AH-SJ-SC:23}. As a third line of research, we aim to implement non-Euclidean contracting
neural networks in machine learning problems akin to methods from~\cite{LK-ME-JJES:22}.
More broadly, we believe that our non-Euclidean contraction framework for
continuous-time NNs serves as a first step to analyzing robustness and convergence
properties of other classes of neural circuits
and other machine learning architectures.

\appendices

\section{}\label{app:proofs}
\subsection{Proof of Theorem~\ref{thm:quasiconvexity}}
\begin{proof}
	  Regarding statement~\ref{item:quasiconvexity}, we provide the proof for
	$p = 1$ since $p = \infty$ is essentially identical. Continuity is a straightforward consequence of the formula for $\mu_{1,[\eta]}$. Regarding quasiconvexity, we will show that
	sublevel sets of the map $\eta \mapsto \mu_{1,[\eta]}(A)$ are convex. For fixed $b \in \real$,
	the set $\setdef{\eta \in
		\realpositive^n}{\mu_{1,[\eta]}(A) \leq b}$ is characterized
	by $\eta$ satisfying
	$$\eta_i A_{ii} + \sum\nolimits_{j=1,j\not=i}^n \eta_j |A_{ji}| \leq
	\eta_i b, \quad \text{for all } i \in \until{n}.$$ Since each of these
	inequalities is linear in $\eta$, for fixed $b$, the above set is a
	polytope, proving quasiconvexity.  Statement~\ref{item:bisection} follows
	from the definitions of $\mu_{1,[\eta]}(A)$ and
	$\mu_{\infty,[\eta]^{-1}}(A)$.
\end{proof}

\newcommand{\mcD}{\mathcal{D}}
\subsection{Proof of Lemma~\ref{lemma:affine-scaling}}\label{app:worst-case-polytopes}
To prove Lemma~\ref{lemma:affine-scaling}, we first need a technical result.
\begin{lemma}\label{lemma:mzr-equality}
	For any $\gamma \in \real, A \in \real^{n\times n}$, the following holds:
	$$\metzler{\gamma A} = \metzler{|\gamma|A - (|\gamma| -\gamma)(I_n \circ A)}.$$
\end{lemma}
\begin{proof}
	The proof follows by checking that the corresponding entries of each matrix are equal.
\end{proof}
\begin{proof}[Proof of Lemma~\ref{lemma:affine-scaling}]
	First we show~\eqref{fact:ms:1}. Use the short-hand $r_i := A_{ii} + \sum_{j \neq i} |A_{ij}|\eta_i/\eta_j$ and $\mcD := \{\dmin,\dmax\}$. Then
	\begin{align*}
	&\max_{d \in [\dmin,\dmax]^n} \mu_{\infty,[\eta]}([c] + [d]A) \\
	&=\max_{d \in [\dmin,\dmax]^n} \max_{i \in \until{n}} c_i + d_iA_{ii} + \sum_{j \neq i} |d_i A_{ij}|\frac{\eta_i}{\eta_j} \\
	&= \max_{i \in \until{n}}\max_{d \in [\dmin,\dmax]^n} c_i + d_iA_{ii} + \sum_{j \neq i} |d_i A_{ij}|\frac{\eta_i}{\eta_j} \\
	&\overset{\bigstar}{=} \max_{i \in \until{n}} \max\setdef{c_i +\gamma A_{ii} + |\gamma|\sum_{j\neq i} |A_{ij}|\frac{\eta_i}{\eta_j}}{\gamma \in \mcD},
	\end{align*}
	where the equality $\overset{\bigstar}{=}$ holds because the function $d_i \mapsto c_i + d_iA_{ii} + \sum_{j\neq i} |d_i A_{ij}|\eta_i/\eta_j$ is convex. Since the maximum value of a convex function over an interval $d_i \in [\dmin,\dmax]$ occurs at one of the endpoints, the equality $\overset{\bigstar}{=}$ is justified.
	
	Additionally note that for any $\gamma \in \real$,
	\begin{align*}
	\gamma A_{ii} + |\gamma|\sum_{j \neq i}|A_{ij}|\frac{\eta_i}{\eta_j} = |\gamma|r_i - (|\gamma|-\gamma)A_{ii}.
	\end{align*}
	Therefore,
	\begin{align*}
	&\max_{d \in [\dmin,\dmax]^n} \mu_{\infty,[\eta]}([c] + [d]A)\\
	&= \max_{i\in\until{n}} \max\setdef{c_i + |\gamma|r_i - (|\gamma| - \gamma)A_{ii}}{\gamma \in \mcD} \\
	&= \max\setdef{\mu_{\infty,[\eta]}([c] + |\gamma|A - (|\gamma|-\gamma)(I_n \circ A))}{\gamma \in \mcD}, \\
	&\overset{\spadesuit}{=} \max\setdef{\mu_{\infty,[\eta]}([c] + \metzler{|\gamma|A - (|\gamma|-\gamma)(I_n \circ A)})}{\gamma \in \mcD}, \\
	&\overset{\clubsuit}{=} \max\setdef{\mu_{\infty,[\eta]}([c] + \metzler{\gamma A})}{\gamma \in \mcD} \\
	&\overset{\spadesuit}{=} \max\{\mu_{\infty,[\eta]}([c] + \dmin A), \mu_{\infty,[\eta]}([c] + \dmax A)\},
	\end{align*}
	where equalities $\overset{\spadesuit}{=}$ hold by Theorem~\ref{thm:MetzlerHurwitzness}\ref{item:monotone-mu} and the equality $\overset{\clubsuit}{=}$ holds by Lemma~\ref{lemma:mzr-equality}.
	Thus, formula~\eqref{fact:ms:1} is proved.
	
	Regarding formula~\eqref{fact:ms:3}, we compute
	\begin{align*}
	&\max_{d \in [\dmin,\dmax]^n} \mu_{\infty,[\eta]}([c] + A[d]) \\
	&= \max_{d \in [\dmin,\dmax]^n} \max_{i \in \until{n}} c_i + d_iA_{ii} + \sum_{j \neq i} |d_j A_{ij}| \frac{\eta_i}{\eta_j} \\
	&\leq \max_{i \in \until{n}} \max_{d \in [\dmin,\dmax]^n} c_i + d_iA_{ii} + \od \sum_{j \neq i} |A_{ij}|\frac{\eta_i}{\eta_j} \\
	&\overset{\bigstar}{=} \max_{i \in \until{n}} \max\setdef{c_i + \gamma A_{ii} + \od \sum_{j \neq i} |A_{ij}|\frac{\eta_i}{\eta_j}}{\gamma \in \mcD},
	\end{align*}
	where the equality $\overset{\bigstar}{=}$ holds because the function $d_i \mapsto d_iA_{ii} + \od\sum_{j\neq i} |A_{ij}|\eta_i/\eta_j$ is convex. Since the maximum value of a convex function over an interval $d_i \in [\dmin,\dmax]$ occurs at one of the endpoints, the equality $\overset{\bigstar}{=}$ holds.
	
	Additionally, note that for any $\gamma \in \real$,
	\begin{align*}
	\gamma A_{ii} + \od \sum_{j \neq i} |A_{ij}|\frac{\eta_i}{\eta_j} = \od r_i - (\od - \gamma)A_{ii}.
	\end{align*}
	Therefore,
	\begin{align*}
	&\max_{d \in [\dmin,\dmax]^n} \mu_{\infty,[\eta]}([c] + A[d]) \\
	&\leq \max_{i \in \until{n}} \max\setdef{c_i + \od r_i - (\od - \gamma)A_{ii}}{\gamma \in \mcD}\\
	&= \max\setdef{\mu_{\infty,[\eta]}([c] + \od A - (\od-\gamma)(I_n \circ A))}{\gamma \in \mcD}.
	\end{align*}
	To see that this inequality is tight, suppose $\gamma \in \mcD$ satisfies $|\gamma| = \od$. Then let:
	\begin{align*}
	k &\in \argmax_{i \in \until{n}} c_i + \dmin A_{ii} + \sum_{j \neq i} |\gamma A_{ij}|\frac{\eta_i}{\eta_j}, \\
	m &\in \argmax_{i \in \until{n}} c_i + \dmax A_{ii} + \sum_{j \neq i} |\gamma A_{ij}|\frac{\eta_i}{\eta_j}.
	\end{align*}
	Let $\vect{e}_k$ and $\vect{e}_m$ be unit vectors with $1$ in their $k$-th and $m$-th entry, respectively, and define
	\begin{align*}
	d_k = \gamma \vectorones[n] - (\gamma - \dmin)\vect{e}_{k}, \quad
	d_m = \gamma \vectorones[n] - (\gamma - \dmax)\vect{e}_{m}.
	\end{align*}
	Then by construction,
	\begin{equation}\label{eq:optimal-index}
	\begin{aligned}
	&\mu_{\infty,[\eta]}([c] + A[d_k]) = c_{k} + \dmin A_{kk} + \sum_{j \neq k} |\gamma A_{kj}|\frac{\eta_{k}}{\eta_j} \\
	&= c_{k} + (\dmin-\od) A_{kk} + \od A_{kk} + \od\sum_{j \neq k} |A_{kj}|\frac{\eta_{k}}{\eta_j} \\
	&= c_{k} + \od r_{k} - (\od - \dmin) A_{kk} \\
	&= \mu_{\infty,[\eta]}([c] + \od A - (\od - \dmin)(I_n \circ A)).
	\end{aligned}
	\end{equation}
	Analogously, we have that $\mu_{\infty,[\eta]}([c] + A[d_m]) = \mu_{\infty,[\eta]}([c] + \od A - (\od - \dmax)(I_n \circ A))$.
	
	Additionally, we see
	\begin{align*}
	&\max_{d \in [\dmin,\dmax]^n} \mu_{\infty,[\eta]}([c] + A[d]) \\
	&\geq \max\{\mu_{\infty,[\eta]}([c] + A[d_k]), \mu_{\infty,[\eta]}([c] + A[d_m])\} \\
	&\overset{\eqref{eq:optimal-index}}{=}\max\{\mu_{\infty,[\eta]}([c] + \od A - (\od - \dmin)(I_n \circ A)), \\
	&\qquad\qquad \mu_{\infty,[\eta]}([c] + \od A - (\od - \dmax)(I_n \circ A)) \}.
	\end{align*}
	
	The proofs for~\eqref{fact:ms:2} and~\eqref{fact:ms:4} are straightforward applications of the fact that $\mu_{1,[\eta]}(B) = \mu_{\infty,[\eta]^{-1}}(B^\top)$ and by applying~\eqref{fact:ms:1} and~\eqref{fact:ms:3}, respectively.
\end{proof}

\begin{corollary}[Some simplifications]
	Using the same notation as in Lemma~\ref{lemma:affine-scaling}, suppose
	\begin{enumerate}
		\item $\od = \dmax$ (note that this implies $\dmax \geq 0$). Then
		\begin{multline}
		\max_{d \in [\dmin,\dmax]^n} \mu_{\infty,[\eta]}([c] + A[d]) = \max\{\mu_{\infty,[\eta]}([c] + \dmax A), \\
		\qquad\qquad \mu_{\infty,[\eta]}([c] + \dmax A - (\dmax-\dmin)(I_n \circ A))\}\nonumber.
		\end{multline}
		\item $\od = -\dmin$ (note that this implies $\dmin \leq 0$). Then
		\begin{multline}
		\max_{d \in [\dmin,\dmax]^n} \mu_{\infty,[\eta]}([c] + A[d]) = \max\{\mu_{\infty,[\eta]}([c] + \dmin A), \\
		\qquad\qquad \mu_{\infty,[\eta]}([c] + \dmin A - (\dmin-\dmax)(I_n \circ A))\}\nonumber.
		\end{multline}
	\end{enumerate}
\end{corollary}

\subsection{Proof of Theorem~\ref{thm:MetzlerHurwitzness}}
	\begin{proof}
		From~\cite[Theorem~8.1.18]{RAH-CRJ:12}, for all $A\in\real^{n\times n}$, we have
		\begin{equation}
			\label{eq:monotonicity-rho}
			\rho(A)\leq\rho(|A|),
		\end{equation}
		where $\rho$ denotes the spectral radius of a matrix.
		Regarding statement~\ref{item:monotone-alpha}, pick
		$\gamma>\max_{i}|A_{ii}|$ and define $\bar{A}=A+\gamma I_n$ so that
		$|\bar{A}|=\metzler{A}+\gamma{I_n}$.  We note that
		$\alpha(\bar{A})\leq\rho(\bar{A})$ (which is true for any matrix) and,
		from inequality~\eqref{eq:monotonicity-rho}, we know
		\begin{equation}
			\begin{aligned}
				\alpha(A)+\gamma &=
				\alpha(\bar{A})\leq\rho(\bar{A})\leq\rho(|\bar{A}|)=\alpha(|\bar{A}|)\\&=\alpha(\metzler{A}+\gamma{I_n})=
				\alpha(\metzler{A})+\gamma.
			\end{aligned}
		\end{equation}
		Here $\rho(|\bar{A}|)=\alpha(|\bar{A}|)$ follows from the
		Perron-Frobenius Theorem for non-negative matrices. This proves
		statement~\ref{item:monotone-alpha}.
		
		Regarding statement~\ref{item:monotone-mu}, note that the norm
		$\norm{\cdot}{p,[\eta]}$ is monotonic, it is easy to see that, for all
		matrices $B$, we have $\norm{B}{p,[\eta]}\leq\norm{|B|}{p,[\eta]}$.  For
		small $h>0$, we note $|I_n+h A|= I_n+h\metzler{A}$ so that
		\begin{align*}
			\norm{I_n+h A}{p,[\eta]} & \leq \norm{|I_n+h A|}{p,[\eta]} = \norm{I_n+h\metzler{A}}{p,[\eta]}.
		\end{align*}
		Therefore, for small enough $h > 0$,
		$$\frac{\|I_n + hA\|_{p,[\eta]} - 1}{h} \leq \frac{\|I_n + h\metzler{A}\|_{p,[\eta]} - 1}{h}.$$
		Thus, statement~\ref{item:monotone-mu} follows from the formula~\eqref{eq:lognorm} in the limit as $h\to0^+$. For $p \in \{1,\infty\}$, statement~\ref{item:monotone-mu} holds by the formulas for $\mu_{1,[\eta]}$ and $\mu_{\infty,[\eta]}$.
		
		Finally, regarding statement~\ref{item:optimaldiagonalweights}, for $p \in \{1,\infty\}$, by statement~\ref{item:monotone-mu} we have
		$\inf_{\eta \in \realpositive^n} \mu_{p,[\eta]}(A) = \inf_{\eta \in \realpositive^n} \mu_{p,[\eta]}(\metzler{A}).$
		Moreover, since $\metzler{A}$ is Metzler, by Lemma~\ref{lemma:efficientlognorm-metzler},
		$\inf_{\eta \in \realpositive^n} \mu_{p,[\eta]}(\metzler{A}) = \alpha(\metzler{A})$.
	\end{proof}

\subsection{Proofs of Lemma~\ref{lemma:submatrices} and Corollary~\ref{cor:pruning}}
		\begin{proof}[Proof of Lemma~\ref{lemma:submatrices}]
			
			Regarding statement~\ref{fact:submatrix-norm-bound}, let $D_\mcI$ denote the
			diagonal matrix with entries $(D_\mcI)_{ii}=1$ if $i\in\mcI$ and
			$(D_\mcI)_{ii}=0$ if $i\not\in\mcI$.

			With this notation, we are ready to compute
			\begin{align}
				\norm{A_\mcI}{\mcI} &=
				\max_{y\in\real^{|\mcI|}, \norm{y}{\mcI}=1}   \norm{A_\mcI \, y}{\mcI}
				\qquad &&\text{} \\
				&=
				\max_{y\in\real^{|\mcI|}, \norm{y}{\mcI}=1} \norm{ \operatorname{pad}_\mcI(A_\mcI y)}{}
				\qquad &&
				\\
				&=
				\max_{y\in\real^{|\mcI|}, \norm{\operatorname{pad}_\mcI(y)}{}=1} \norm{(D_\mcI A D_\mcI) \, \operatorname{pad}_\mcI(y)}{}
				\qquad &&
				\\
				&\leq
				\max_{x\in\real^{n}, \norm{x}{}=1} \norm{(D_\mcI A D_\mcI) \, x}{}
				\qquad &&
				\\
				&=  \norm{ D_\mcI A D_\mcI }{} \leq
				\norm{ D_\mcI }{} \norm{A}{} \norm{ D_\mcI  }{}
				= \norm{A}{}.
			\end{align}
			The last equality holds because the monotonicity of $\norm{\cdot}{}$
			implies $\norm{ D_\mcI }{}=1$.  This concludes the proof
			of~\ref{fact:submatrix-norm-bound}.
			
			Statement~\ref{fact:submatrix-mu-bound} follows from the definition of
			log norm and applying statement~\ref{fact:submatrix-norm-bound}
			to the matrix $I_{|\mcI|}+h A_{\mcI}$ as a principal submatrix of
			$I_n+h A$:
			\begin{align*}
				\mu_{\mcI}(A_\mcI) &:=
				\lim_{h\to0^+} \frac{\norm{I_{|\mcI|}+h A_{\mcI}}{\mcI}-1}{h}
				\\ &\leq  \lim_{h\to0^+} \frac{\norm{I_{n}+h A}{}-1}{h} = \mu(A).
			\end{align*}
			Finally, statement~\ref{fact:diagonalmu-totallyH} is an immediate
			consequence of~\ref{fact:submatrix-mu-bound}.
		\end{proof}
	\begin{proof}[Proof of Corollary~\ref{cor:pruning}]
		Regarding item~\ref{item:totally-contracting}, since $A \in \mcMH$, $\alpha(\metzler{A}) < 0$. By Lemma~\ref{lemma:efficientlognorm-metzler}
		and Theorem~\ref{thm:MetzlerHurwitzness}\ref{item:monotone-mu}, for sufficiently small $\epsilon > 0$,
		there exists $\eta \in \realpositive^n$ such that $\lognorm{A}{1,[\eta]} =
		\lognorm{\metzler{A}}{1,[\eta]} \leq \alpha(\metzler{A}) + \epsilon < 0$. Then by
		Lemma~\ref{lemma:submatrices}\ref{fact:submatrix-mu-bound}, for non-empty
		$\mcI \subset \until{n}$, $\mu_{\mcI,1,[\eta]}(A_\mcI) \leq \lognorm{A}{1,[\eta]}
		< 0$. Moreover, by Theorem~\ref{thm:MetzlerHurwitzness}\ref{item:monotone-mu},
		$\alpha(\metzler{A_{\mcI}}) \leq \mu_{\mcI,1,[\eta]}(\metzler{A_\mcI}) = \mu_{\mcI,1,[\eta]}(A_\mcI)
		< 0$. We conclude that $A_\mcI \in \mcMH$.
		Regarding item~\ref{item:conn-contracting}, note that $A - A_{ij}\vect{e}_{ij}$ is the matrix $A$ with its $ij$-th entry zeroed out. Then since $A \in \mcMH$, for sufficiently small $\epsilon > 0$, there exists $\eta \in \realpositive^n$ such that $\mu_{1,[\eta]}(A) < 0$. The result is then a consequence of the fact that $\alpha(\metzler{A - A_{ij}\vect{e}_{ij}}) \leq \mu_{1,[\eta]}(A-A_{ij}\vect{e}_{ij}) \leq \mu_{1,[\eta]}(A) < 0$.
	\end{proof}

\subsection{Proof of Theorem~\ref{theorem:osLequivalence}}
To prove Theorem~\ref{theorem:osLequivalence}, we first recall Clarke's generalized Jacobian from nonsmooth analysis.
\begin{definition}[{\cite[Definition~2.6.1]{FHC:83}}]
	Let $\map{f}{U}{\real^m}$ be locally Lipschitz on an open set $U \subseteq \real^n$ and let $\Omega_{f} \subset U$ be the set of points where $f$ is not differentiable. Then \emph{Clarke's generalized Jacobian} at $x$ is
	\begin{equation}
		\partial f(x) = \conv\setdef{\lim_{i \to \infty} \jac{f}(x_i)}{x_i \to x \text{ and } x_i \notin \Omega_f }.
	\end{equation}
\end{definition}

The mean-value theorem has the following generalization for locally Lipschitz functions. For any two points $x,y \in \real^n$, denote $[x,y] := \setdef{tx + (1 - t)y}{t \in [0,1]}$.
\begin{lemma}\label{lem:MVTgeneralization}
	For $\map{f}{U}{\real^m}$ locally Lipschitz on an open convex set $U \subseteq \real^n$, let $[x,y]\subset U$. Then 
there exists matrix $A\in\real^{m\times n}$ such that
	\begin{equation}\label{eq.mean-value}
		f(x)-f(y)=A(x-y)\;\text{and}\; A\in\conv\bigcup_{u\in[x,y]}\partial f(u).
	\end{equation}
\end{lemma}
\begin{proof}
	Since $[x,y]$ is a compact subset of $\real^n$, it can be easily seen that there exists a convex open set $U_0$
	such that $[x,y]\subset U_0\subset\bar U_0\subset U$ (where $\bar U_0$ is the closure of $U_0$).
	The statement now follows from~\cite[Proposition~2.6.5]{FHC:83}.
\end{proof}
\begin{proof}[Proof of Theorem~\ref{theorem:osLequivalence}]
	Regarding \ref{item:nonsmoothosL} $\implies$ \ref{item:innerprodosL}, Let $x,y \in U$. Since
	$U$ is convex, $[x,y] \subset U$. Then by Lemma~\ref{lem:MVTgeneralization}, there exists 
$A$ satisfying the conditions~\eqref{eq.mean-value}. Condition~\ref{item:nonsmoothosL} in 
implies that $\mu(\cdot)$ does not exceed $c$ on each set $\partial f(u)$ by continuity and convexity of $\mu$, entailing that
$\mu(A) \leq c$. Therefore,
	\begin{align*}
		\WeakP{f(x) - f(y)}{x - y} &= \WeakP{A(x - y)}{x - y} \\&\leq \mu(A)\|x - y\|^2 \leq c\|x - y\|^2,
	\end{align*}
	where the second line is due to Lumer's equality, Lemma~\ref{lemma:Lumer}.
	Regarding \ref{item:innerprodosL} $\implies$ \ref{item:nonsmoothosL}, let $x \in U$ such
	that $\jac{f}(x)$ exists and let $v \in \real^n$ and $h > 0$. Then by assumption,
	\begin{multline*}
		\WeakP{f(x+hv) - f(x)}{hv} \leq c\|hv\|^2 \\
		\implies h\WeakP{f(x+hv) - f(x)}{v} \leq ch^2\|v\|^2,
	\end{multline*}
	which holds by the weak homogeneity of the weak pairing. Dividing by $h^2 > 0$ and taking the limit as $h \to 0$ implies
	\begin{align*}
		\lim_{h\to0^+} & \WeakP{\frac{f(x+hv) - f(x)}{h}}{v} \leq c\|v\|^2 \\
		&\implies \WeakP{\jac{f}(x)v}{v} \leq c\|v\|^2 \\
		&\implies \mu(\jac{f}(x)) \leq c,
	\end{align*}
	where the final implication holds by taking the supremum over all $v \in \real^n$ with $\|v\| = 1$ together with Lumer's equality, Lemma~\ref{lemma:Lumer}. Therefore, statement~\ref{item:nonsmoothosL} holds.
\end{proof}
\end{arxiv}
\bibliographystyle{plainurl+isbn}
\bibliography{alias,Main,FB,New}
\end{document}